\documentclass[10pt]{amsart}

 \newtheorem{thm}{Theorem}[subsection]
 \newtheorem{cor}[thm]{Corollary}
 \newtheorem{lem}[thm]{Lemma}
 \newtheorem{prop}[thm]{Proposition}
 \theoremstyle{definition}
 \newtheorem{defn}[thm]{Definition}
 \theoremstyle{definition}
 \newtheorem{example}[thm]{Example}
 \newtheorem{rem}[thm]{Remark}

\theoremstyle{definition}

 \DeclareMathOperator{\R}{\mathbb{R}}
\DeclareMathOperator{\D}{\mathcal{D}}
\DeclareMathOperator{\M}{\mathcal{M}}
 
 \DeclareMathOperator{\C}{\mathbb{C}}

 \DeclareMathOperator{\MAJ}{maj}
\DeclareMathOperator{\SP}{Sp}

\newcommand{\Y}{\mathbb{Y}}
\newcommand{\lm}{\lambda}
\newcommand{\Lm}{\Lambda}

\begin{document}

\title[A model for the deformed Plancherel process.]
 {\bf{A differential model for the deformation of the Plancherel
 growth process}}

\author{Eugene  Strahov}
\address{ Department of Mathematics, The Hebrew University of
Jerusalem, Givat Ram, Jerusalem
91904}\email{strahov@math.huji.ac.il}

\begin{abstract}

In the present paper  we construct and solve a differential model
for the $q$-analog of the Plancherel growth process. The
construction is based on a deformation of the Makrov-Krein
correspondence between continual diagrams and probability
distributions.

\end{abstract}

\maketitle
\section{Introduction}
The Plancherel growth of Young diagrams has been the subject of an
intensive research for many years (see, for example, the book by
Kerov \cite{kerov1}, and also the expository article by Vershik
\cite{vershik0} for a recent review). An important result in the
field  is the asymptotics of the shape of the Young diagram (called
the limit shape) in the course of the Plancherel growth process
(Logan and Shepp \cite{logan}, Vershik and Kerov \cite{vershik11}).
In \cite{kerov02} Kerov constructed a dynamical model for the
Plancherel growth process, and showed that the limit shape is a
fixed point of the Burgers equation: it attracts asymptotically all
solutions of the Burgers equation of a certain class. The
construction of the dynamical model is based on the correspondence
between continual diagrams and probability distributions, called by
Kerov the Markov-Krein correspondence.

The Plancherel measure admits a natural deformation as it follows
from the representation theory of the Iwahori-Hecke algebras. The
deformed Plancherel measure defines a stochastic process which is a
natural $q$-analog of the Plancherel growth.  The main goal of the
present paper is to construct and to solve a differential model for
this process. The method  is in a deformation of the Markov-Krein
correspondence, which results in a deformation  of the differential
equations responsible for the dynamic of continual diagrams. It is
shown in this paper that the deformed Burgers equation has a fixed
point, and it is proved that the fixed point attracts asymptotically
all relevant solutions.
\subsection{Background and remarks on the related works}
\subsubsection{The Plancherel growth process} Let $\Y$ denote the
Young graph, and let $\Y_n$ be the level of $\Y$ consisting of the
Young diagrams with $n$ boxes. Thus $\Y=\bigcup_{n=0}^{\infty}\Y_n$.
 By
definition, the Plancherel growth process is the Markov chain on
$\Y$ whose initial state is the empty diagram, and whose transition
probabilities $p(\lambda,\Lambda)$ are given by
$$
p(\lambda,\Lambda)=\frac{1}{|\Lambda|}\frac{\dim\Lambda}{\dim\lambda},
$$
if $\lambda$ is obtained from $\Lambda$ by removing one box, and by
$p(\lambda,\Lambda)=0$ otherwise. Here $\dim\lambda$ denotes the
number of standard Young diagrams of shape $\lambda$, and
$|\lambda|$ denotes the number of boxes in $\lambda$. It can be
shown (see, for example, Kerov \cite{kerov1}) that the distribution
$M_n$ of the state $\lambda\in\Y_n$ coincides with the Plancherel
measure on $\Y_n$, $M_n(\lambda)=\frac{\dim^2\lambda}{|\lambda|!}$.
If the Young diagrams on each level of $\Y$ are distributed
according to the Plancherel measure then the Plancherel growth
process is defined. It is known ( Vershik and Kerov
\cite{vershik11}, Logan and Shepp \cite{logan}) that in the course
of the Plancherel growth processes almost all Young diagrams with
the normalized area become uniformly  close to a common universal
curve. In a natural coordinate system this curve is given by
\begin{equation}\label{OOOOMMMMEEEGA}
\Omega(s)=\left\{%
\begin{array}{ll}
    \frac{2}{\pi}\left(s\arcsin\frac{s}{2}+\sqrt{4-s^2}\right), & \hbox{if} \;\; |s|\leq 2, \\
    |s|, & \hbox{if} \;\;|s|\geq 2. \\
\end{array}%
\right.
\end{equation}
\subsubsection{A formula for the transition probabilities}
Given $\lambda\in\Y$ define a piecewise linear function $\lambda(s)$
with slopes $\pm 1$ and local minima and maxima  at two interlacing
sequences of integer points
$$ x_1<y_1<x_2<\ldots<x_m<y_m<x_{m+1},$$
where the $x_i$'s are the local minima, and the $y_i$'s are  the
local maxima of $\lambda(s)$, see Figure 1.
\begin{center}
\begin{figure}[h]
\setlength{\unitlength}{4pt}
\begin{picture}(80,50)
\put(40,0){\vector(0,20){40}} \put(40,0){\line(1,1){40}}
\put(40,0){\line(-1,1){40}} \put(36,4){\line(1,1){20}}
\put(56,24){\line(1,-1){4}} \put(44,4){\line(-1,1){24}}
\put(48,8){\line(-1,1){16}} \put(52,12){\line(-1,1){12}}
\put(56,16){\line(-1,1){4}} \put(32,8){\line(1,1){12}}
\put(28,12){\line(1,1){12}} \put(24,16){\line(1,1){8}}
\put(20,20){\line(1,1){4}} \put(16,24){\line(1,1){4}}
\multiput(16,24)(0,-1){25}%
{.}
\put(13,1){$x_1$} \multiput (28,19)(0,-1){20}%
{.} \put(25,1){$x_2$}
 \multiput(20,27)(0,-1){28}%
 {.}
\put(19,1){$y_1$} \multiput(60,19)(0,-1){20}%
{.}
 \put(61,1){$x_{m+1}$}
\multiput(56,23)(0,-1){24}%
{.} \put(53,1){$y_{m}$} \put(-5,0){\vector(1,0){90}}
\end{picture}
\caption{Young diagrams as interlacing sequences}
\end{figure}
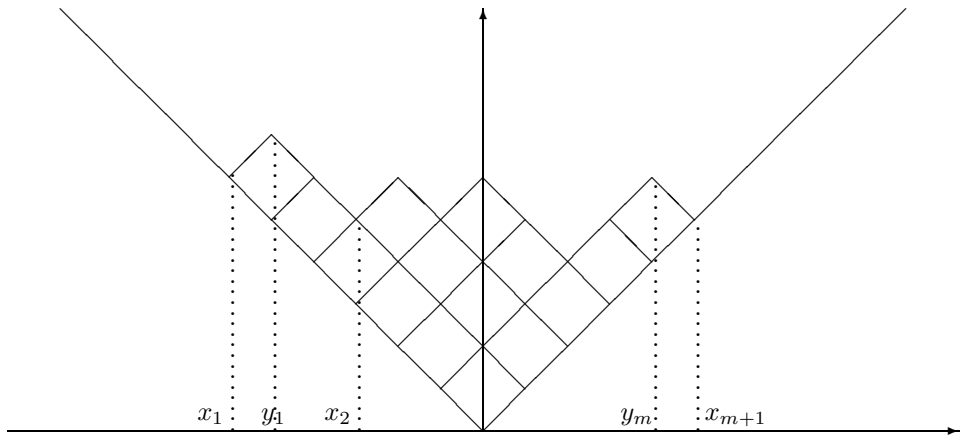
\end{center}
Write $\mu_k(\lambda)$ instead of $p(\lambda,\Lambda)$ if the box
that distinguishes $\Lambda$ from $\lambda$ is attached to the
minimum of the function $\lambda(s)$ with the coordinate $x_k$. Then
$\{\mu_k(\lambda)\}_{k=1}^{m+1}$ are precisely the coefficients of
the partial fraction expansion
\begin{equation}\label{YYTT}
\sum\limits_{k=1}^{m+1}\frac{\mu_k(\lambda)}{x-x_k}=\frac{\prod\limits_{i=1}^{m}(x-y_i)}{\prod\limits_{i=1}^{m+1}(x-x_i)},
\end{equation}
see Kerov \cite{kerov0,kerov01}.
 Formula \ref{YYTT} determines the one-to-one correspondence
between the set of Young diagrams, and the set of discrete
probability distributions.
\subsubsection{The Markov-Krein
correspondence.} Kerov showed in \cite{kerov0,kerov1} that the
correspondence between Young diagrams and discrete probability
distributions defined by equation (\ref{YYTT})  can be extended by
continuity. This extension  was called by Kerov the Markov-Krein
correspondence. More precisely, there is   a bijective
correspondence $\mu\longleftrightarrow w$ between the set of
probability measures on $\R$ with compact support, and the set of
continual diagrams (the definition of continual diagrams can be
found in section \ref{SECTIONCONTDIAGRAMS}). It is characterized by
the relation
$$
R_{\mu}(x)=R_{w}(x),
$$
where
$$
R_{\mu}(x):=\int\frac{\mu(ds)}{x-s},\;\;
R_{w}(x):=\frac{1}{x}\exp\left[-\int\frac{\sigma'(s)ds}{s-x}\right],
$$
 $x\in\C/I$, $I\subset\R$ stands for a sufficiently large
interval, and $\sigma(s)=\frac{1}{2}(w(s)-|s|)$. The function
$R_{w}(x)$ is called the $R$-function of the diagram $w$, and the
function $R_{\mu}(x)$ is called the $R$-function of the measure
$\mu$. If equation $R_{\mu}(x)=R_{w}(x)$ is satisfied the measure
$\mu$ is referred to as the transition distribution of the diagram
$w$.

Let $w$ be a continual diagram, and define the function $F(s)$ by
the formula
$$
F(s)=\frac{1}{2}\left(1+w'(s)\right).
$$
$F(s)$ can be regarded as the distribution function of a signed
measure $\tau$, and $\tau$ is referred to as the Rayleigh measure of
the continual diagram $w$. The Markov-Krein correspondence turns
into the relationship between a probability distribution $\mu$, and
a bounded signed measure $\tau$ on the real line satisfying the
identity
\begin{equation}\label{MarkovKreinIntr}
\int\frac{\mu(ds)}{z-s}=\exp\int\ln\frac{1}{z-s}\tau(ds).
\end{equation}
To see the relation with transition probabilities of the Plancherel
growth process assume that $\mu$ is the discrete distribution with
weights $\mu_k(\lambda)$ at the points $x_k$, where
$\left\{x_k\right\}_{k=1}^{m+1}$ are the local minima of the
function $\lambda(s)$, see Figure 1. Let $\tau$ be the signed
measure with the weights $+1$ at the points
$\left\{x_i\right\}_{i=1}^{m+1}$, and the weights $-1$ at the points
$\left\{y_i\right\}_{i=1}^{m}$. Then identity
(\ref{MarkovKreinIntr}) specializes to (\ref{YYTT}).

  Besides its
relevance to  the Plancherel growth process the Markov-Krein
correspondence plays a role in diverse topics in
analysis including\\
(1) the connection between additive  and multiplicative integral
 representations of analytic functions of negative imaginary
 type;\\
(2) the Markov moment problem;\\
(3) distributions of mean values of Dirichlet random measures;\\
(4) the theory of spectral shift function in the scattering
theory,\\
see the expository paper by Kerov  \cite{kerov03}, where a variety
of applications of the Markov-Krein correspondence are described.
Note also that more general versions of the Markov-Krein
correspondence were used in Kerov and Tsilevich \cite{kerov3} in
connection with the Dirichlet measures,  and in Vershik, Yor, and
Tsilevich \cite{vershik3}.

\subsubsection{Differential model for the Plancherel growth of
Young diagrams}\label{KGROWTMODEL} Kerov showed in \cite{kerov02}
that the limiting diagram $\Omega(s)$  is a fixed point of the
Burgers equation (equation (\ref{INTRBD1}) below), i.e. $\Omega(s)$
attracts asymptotically all solutions of this equation. The Burgers
equation naturally arises in the framework of the differential model
for the  growth of Young diagrams constructed in \cite{kerov02}. The
main assumptions behind this
model are:\\
1) The history of a growth of a continual diagram $w$ is described
by a curve $w(.,t), t_0<t<\infty$ in the space of continual
diagrams. The diagrams $w(.,t)$ are assumed to increase (with
respect to the inclusion of subgraphs)
with $t$.\\
2) The diagram $w(s,t)$ is required to grow in the direction of its
transition distribution $\mu_t$, which means ( see \cite{kerov02})
$$
\mu_t(ds)=\frac{\partial\sigma(s,t)}{\partial t}ds.
$$
The equation above is called the basic dynamic equation.
 The bijection between the continual diagrams
and the probability measures leads to different equivalent forms of
the basic dynamic equation:
\begin{equation}\label{INTRBD}
\int\frac{1}{x-s}\frac{\partial\sigma(s,t)}{\partial
t}ds=\frac{1}{x}\exp\left[-\int\frac{1}{x-s}\frac{\partial\sigma(s,t)}{\partial
s}ds\right]
\end{equation}
\begin{equation}
\frac{d}{dt}p_n(t)=(n+1)(n+2)h_n(t),\;\; n=0,1,\ldots .
\end{equation}
\begin{equation}\label{INTRBD1}
\frac{\partial R(x;t)}{\partial t}+R(x,t)\frac{\partial
R(x,t)}{\partial x}=0,
\end{equation}
where
$$
R(x,t)=\int\frac{\mu_{w(.,t)}(ds)}{x-s}=\int\frac{\partial}{\partial
t}\sigma(s,t)\frac{ds}{x-s},
$$
$p_n(t)$ are the moments of a diagram $w(s,t)$, and  $h_n(t)$ be the
moments of its transition distribution $\mu_{w(.,t)}(ds)$. The main
result in
\cite{kerov02} is the following \\
\begin{thm}\label{THEOREMA}
Assume that the function $\sigma(s,t)=\left(w(s,t)-|s|\right)/2$
satisfies  equation (\ref{INTRBD}) (which is equivalent to the basic
dynamic equation via the Krein correspondence, and to the Burgers
equation, equation (\ref{INTRBD1})). Then
$$
\underset{t\rightarrow\infty}{\lim}\frac{1}{\sqrt{t}}w(s\sqrt{t},t)=\Omega(s)
$$
uniformly in $s\in\R$, where $\Omega(s)$ is given by equation
(\ref{OOOOMMMMEEEGA}).
\end{thm}
Theorem \ref{THEOREMA} means that all solutions of equation
(\ref{INTRBD}) have the common asymptotics as $t\rightarrow\infty$.
Therefore the differential model of the growth described above is a
continuous time deterministic process with the same asymptotic
behavior as the (random) Plancherel growth process.

Equation (\ref{INTRBD1}) is a quasi-linear differential equation
which describes the free motion of a one-dimensional medium of
noninteracting particles \cite{kerov02}. In terms of equation
(\ref{INTRBD1}) the curve $\Omega(s)$ corresponds to the automodel
solution
$$
R(x,t)=\frac{1}{\sqrt{t}}r(\frac{x}{\sqrt{t}}),
$$
where $r(x)$ satisfies the nonlinear differential equation
\begin{equation}\label{malrdif}
2rr'-xr'-r=0.
\end{equation}
The only solution of this equation vanishing at $x\rightarrow\infty$
is
$$
r(x)=\frac{1}{2}(x-\sqrt{x^2-4}).
$$
This is precisely the $R$-function of the diagram $\Omega(s)$.
\subsection{Statement of main results} We start with a natural
$q$-deformation of the Plancherel measure $M_q^{(n)}$ (equation
(\ref{-15})) which is originated from the representation theory of
the Iwahori-Hecke algebras (sections \ref{IwHSec}, \ref{IwHSecII}).
It is shown in section \ref{IwHSecIIIII} how $M_q^{(n)}$ is related
with non-uniform random permutations. $M_q^{(n)}$ defines a
$q$-analog of the Plancherel growth process, which is  a Markov
chain on the Young graph.  The transition probabilities of this
Markov chain can be described as follows. Suppose that the Young
diagrams $\lambda$ and $\Lambda$ are distinguished by one box
attached to the minimum of the function $\lambda(s)$ with the
coordinate $x_k$, see Figure 1. Then the probability of the
transition from $\lambda$ to $\Lambda$ is denoted by
$\mu_k(\lambda;q)$.  $\mu_k(\lambda;q)$ satisfies the equation
\begin{equation}\label{qqqqqqqqqqqqqqqqqqq}
\sum\limits_{k=1}^{m+1}\frac{\mu_k(\lambda;q)}{1-q^{x-x_k}}=\frac{\prod_{k=1}^m(1-q^{x-y_k})}{\prod_{k=1}^{m+1}(1-q^{x-x_k})}
\end{equation}
for sufficiently large values of the parameter $x$. Note that as $q$
approaches 1 equation (\ref{qqqqqqqqqqqqqqqqqqq}) is reduced to
equation (\ref{YYTT}) for the transition probabilities in the
Plancherel growth process. Therefore (\ref{qqqqqqqqqqqqqqqqqqq})
defines transition probabilities for a $q$-analog of the Plancherel
growth process.

The main goal of the present paper is to construct a differential
model for this growth process.  For this purpose we  introduce the
$q$-deformation of the Krein correspondence between continual
diagrams and probability measures with compact supports.
\subsubsection{The $q$-deformation of the Markov-Krein
correspondence} Let $0<q\leq 1$, and assume that a real variable $x$
takes values outside an interval $[a,b]$. Denote by $\D[a,b]$ the
set of continual diagrams with the property $w(s)=|s-s_0|$ for
$s\notin [a,b]$. In addition, denote by $\M[a,b]$ the space of
probability measures on the interval $[a,b]$. For $0<q<1$ the
$q$-deformation of the $R$-function of a continual diagram
$w\in\D[a,b]$ is defined by the expression
\begin{equation}
R_{w}(x;q)=\frac{1-q}{1-q^x}\exp\left[-\ln
q^{-1}\int\limits_a^b\frac{d\sigma(s)}{1-q^{x-s}}\right]
=\frac{1-q}{1-q^x}\exp\left[-\frac{1}{2}\ln
q^{-1}\int\limits_a^b\frac{d\left(w(s)-|s|\right)}{1-q^{x-s}}\right],
\nonumber
\end{equation}
and the $q$-deformation of the $R$-function of a measure
$\mu\in\M[a,b]$ is defined by the expression
\begin{displaymath}
R_{\mu}(x;q)=(1-q)\int\limits_a^b\frac{\mu(ds)}{1-q^{x-s}}.
\end{displaymath}
 For $q=1$ the
$q$-deformation of the  $R$-function of a diagram $w\in\D[a,b]$ is
defined to be  $R_{w}(x)$, and the $q$-deformation of the
$R$-function of a measure $\mu$ is defined to be $R_{\mu}(x)$.
\begin{thm}\label{RRRTTT}
Let $q$ be a fixed parameter which takes values in the interval
$(0,1]$. The relation $R_{\mu_q}(x;q)=R_{w(.;q)}(x;q)$ defines the
one-to-one correspondence $w_q\longleftrightarrow\mu_q$ between
continual diagrams from $\D[a,b]$, and the probability measures from
$\M[a,b]$.
\end{thm}
If equation $R_{\mu_q}(x;q)=R_{w(.;q)}(x;q)$ is satisfied, then
$\mu_q$ is referred to as the $q$-transition measure of the diagram
$w(.;q)$. An equivalent form of theorem \ref{RRRTTT} is
 \begin{thm}
There is a relationship between a probability measure $\mu_q$ on
$[a,b]$ and a Rayleigh measure $\tau_q$ on $[a,b]$ defined by the
identity
\begin{equation}\label{qMark}
\int\limits_a^b\frac{\mu_q(ds)}{1-q^{x-s}}=\exp\left[\int\limits_a^b\ln\left(\frac{1}{1-q^{x-s}}\right)\tau_q(ds)\right]
\end{equation}
The probability measure $\mu_q$ and the Rayleigh measure $\tau_q$
determine each other uniquely.
\end{thm}
Equation (\ref{qMark}) can be considered as the $q$-deformation of
the Markov-Krein correspondence (\ref{MarkovKreinIntr}).
 Let us emphasize that in the equality
$R_{\mu_q}(x;q)=R_{w(.;q)}(x;q)$ both the diagram, $w(.;q)$, and the
measure, $\mu_q$,  generally depend on the parameter $q$. Assume
that the interval $[a,b]$ is chosen to be large enough, and that
$\mu_q$ is the discrete distribution with weights $\mu_k(\lambda;q)$
at the points $x_k$, where $\{x_k\}_{k=1}^{m+1}$ are the local
minima of the function $\lambda(s)$, see Figure 1. Let $\tau_q$ be
the signed measure with the weights $+1$ at the points
$\{x_i\}_{i=1}^{m+1}$, and the weights $-1$ at the points
$\{y_i\}_{i=1}^{m}$. Then equation (\ref{qMark}) specializes to
(\ref{qqqqqqqqqqqqqqqqqqq}). In this work we apply  (\ref{qMark}) to
derive the differential model for the $q$-analog of the Plancherel
growth process. It is of interest to investigate the role of
(\ref{qMark}) in other topics of analysis, but we leave this issue
for the future research.
\subsubsection{The differential model for the $q$-analog of the
Plancherel growth process} We start from the same assumptions as in
the Kerov growth model, see section \ref{KGROWTMODEL}. Thus the
history of the growth of a continual diagram $w(.;q)$ is described
by a curve $w(.,t;q)$, $t_0<t<\infty$, the diagram $w(.,t;q)$ is
assumed to increase, and to grow in the direction of its
$q$-transition distribution $\mu_{t,q}$. Introduce the functions
$\left\{h_n[\mu_{t,q};q]\right\}_{n=1}^{\infty}$
\begin{equation}
h_n[\mu_{t,q};q]=\int\limits_a^bq^{-ns}\mu_{t,q}(ds), \nonumber
\end{equation}
and the functions $\left\{p_n[w(.,t;q);q]\right\}_{n=1}^{\infty}$
\begin{equation}
p_n[w(.,t;q);q]=-n\ln
q^{-1}\int\limits_a^bq^{-ns}\;\frac{\partial\sigma(s,t;q)}{\partial
s}\;ds+1. \nonumber
\end{equation}
The theorem below gives the analog of the dynamic equations
(\ref{INTRBD})-(\ref{INTRBD1}):
\begin{thm}\label{THHHH}
 The following dynamic equations
are equivalent
\begin{equation}\label{EEQ}
\int\limits_a^b\left(1-q^{x-s}\right)^{-1}\frac{\partial\sigma(s,t;q)}{\partial
t}ds=\left(1-q^x\right)^{-1}\exp\left[-\ln
q^{-1}\int\limits_a^b\left(1-q^{x-s}\right)^{-1}\frac{\partial\sigma(s,t;q)}{\partial
s}ds\right];
\end{equation}
\begin{equation}
\frac{\partial}{\partial t}p_n\left[w(.,t;q);q\right]
=n^2\ln^2q^{-1}\sum\limits_{|\lambda|=n}\prod\limits_{k=1}^{m(\lambda)}\frac{p_k^{r_k}[w(.,t;q);q]}{k^{r_k}r_k!},
\end{equation}
where $n=1,2,\ldots ,$ $\lambda=\left(1^{r_1},2^{r_2},\ldots
,m^{r_m}\right)$, and $m=m(\lambda)$;
\begin{equation}\label{qRRRQ}
\frac{\partial R_{w(.,t;q)}(x;q)}{\partial x}+\frac{1-q}{\ln
q^{-1}}R_{w(.,t;q)}^{-1}(x;q)\frac{\partial R_{w(.,t;q)}(x;q)
}{\partial t}=0.
\end{equation}
\end{thm}
Partial differential equation (\ref{qRRRQ}) can be understood as a
$q$-analog of the Burgers equation (\ref{INTRBD1}). The main
difference between (\ref{qRRRQ}) and (\ref{INTRBD1}) is that
$R_{w(.,t;q)}(x;q)$ is a function of three variables: $x$, $t$, and
$q$. We remark that the crucial observation behind theorem
(\ref{THHHH}) is that the functions
$\left\{h_n[\mu_{t,q};q]\right\}_{n=1}^{\infty}$ and
$\left\{p_n[w(.,t;q);q]\right\}_{n=1}^{\infty}$ are related to each
other as the generators of the algebra $\Lambda$ of the symmetric
functions, $\{\textbf{h}_n\}_{n=1}^{\infty}$, and
$\{\textbf{p}_n\}_{n=1}^{\infty}$.
\subsubsection{The description of the limiting diagram.}
\begin{thm}
Let $q$ be a fixed parameter which takes values from the open
interval $(0,1)$. Assume that $w(s,t;q)$ is a solution of the
equivalent dynamic equations  (\ref{EEQ})-(\ref{qRRRQ}).\\
\textbf{Claim 1.} There exists a limiting continual diagram
$\Omega(s;q)$ such that
$$
\underset{t\rightarrow\infty}{\lim}\frac{1}{\sqrt{t}}\;w\left(s\sqrt{t},t;q^{\frac{1}{\sqrt{t}}}\right)=\Omega(s;q)
$$
uniformly in $s\in\R$ and $q\in(0,1)$.
\\
\textbf{Claim 2.} The limiting diagram is uniquely determined by the
function $R_{\Omega(.;q)}(x;q)$ defined by
$$
R_{\Omega(.;q)}(x;q)=\left(1-q^x\right)^{-1}\exp\left[-\frac{1}{2}\ln
q^{-1}\int\limits_a^b\frac{d\left(\Omega(s;q)-|s|\right)}{1-q^{x-s}
}ds\right],
$$
and this function, $R_{\Omega(.;q)}(x;q)$, is the solution of the
equation
\begin{equation}\label{FINALSOLUTIONOMEGA}
r(1-q^{x-\frac{\ln q^{-1}}{1-q}r})=1-q.
\end{equation}
\textbf{Claim 3.} Let $\tau^{\Omega(.;q)}$ be the Rayleigh measure
of the limiting diagram $\Omega(s;q)$. For $0<q<1$  the moments
$\left\{p_n[\Omega(.;q);q]\right\}_{n=1}^{\infty}$ of
$\tau^{\Omega(.;q)}$ defined by
$$
p_n[\Omega(.;q);q]=\int\limits_{a}^bq^{-ns}\tau_q^{\Omega(.;q)}(ds)
$$
can be expressed in terms of the solution $\{y_n\}_{n=1}^{\infty}$
of the system of differential equations
\begin{equation}
\frac{dy_n(\varsigma)}{d\varsigma}
=n^2\left\{\sum\limits_{|\lambda|=n}\prod\limits_{k=1}^{m(\lambda)}\frac{y_k^{r_k}(\varsigma)}{k^{r_k}r_k!}\right\},\;\;
n=1,2,\ldots \nonumber
\end{equation}
(where $n=1,2,\ldots $, $\lambda=\left(1^{r_1},2^{r_2},\ldots,
m^{r_m}\right)$, and  $m=m(\lambda)$) are subjected to the initial
conditions $y_n(0)=1,\;n=1,2,\ldots$. Namely,
$$
p_n[\Omega(.;q);q][q]=y_n[\ln^2q],\;\;n=1,2,\ldots.
$$
\end{thm}
The theorem does not provide an explicit form for the limiting curve
$\Omega(s;q)$, but it determines the function
$R_{\Omega(s;q)}(x;q)$, and the moments of the curve $\Omega(s;q)$
explicitly. In terms of equation (\ref{qRRRQ}) the curve
$\Omega(s;q)$ corresponds to the $q$-auto-model solution
$$
R_{w(.,t;q)}(x;q)=\frac{1-q}{1-q^{\sqrt{t}}}r(\frac{x}{\sqrt{t}};q^{\sqrt{t}}),
$$
where $r(x;q)$ satisfies the  partial differential equation
\begin{equation}\label{difmalq}
2r\frac{\partial}{\partial x}r-\frac{1-q}{\ln
q^{-1}}x\frac{\partial}{\partial
x}r-qr-q(1-q)\frac{\partial}{\partial q}r=0.
\end{equation}
Note that if we take $q=1$ in (\ref{difmalq}), then (\ref{difmalq})
is reduced to (\ref{malrdif}). Equation (\ref{difmalq}) can be
reduced to a quasi-linear partial differential equation which can be
solved by the method of characteristics.  The result is that the
solution of (\ref{difmalq}) is uniquely determined by equation
(\ref{FINALSOLUTIONOMEGA}). Clearly,  the solution of
(\ref{FINALSOLUTIONOMEGA}) can be understood as the $q$-deformation
of the $R$-function corresponding to the diagram $\Omega(s)$ defined
by equation (\ref{OOOOMMMMEEEGA}). Indeed, if $q$ approaches $1$ in
equation (\ref{FINALSOLUTIONOMEGA}), then this equation turns into
$r(x-r)=1$. The only solution of this equation vanishing at
$x\rightarrow+\infty$ coincides with the $R$-function of the diagram
$\Omega(s)$.
\section{A deformation of the Plancherel measure}
\subsection{Iwahori-Hecke algebras}\label{IwHSec}
This section recalls few facts on the Iwahori-Hecke algebras
associated with the finite Coxeter groups. The general references on
the Hecke algebras are Curtis and Reiner \cite{curtis}, $\S$67-68,
Carter \cite{carter}, $\S$10.8-10.11. We follow the presentation in
the paper by Diaconis and Ram \cite{diaconis}, which contains the
necessary representation theoretic background (sections 3 and 7).

Let $W$ be a finite Coxeter group generated by simple reflections
$s_1,\ldots, s_n$. A choice of reflection generators gives rise to a
length function $l$ on a Coxeter group. $l$ is defined as the
minimum number of the reflection generators required to express a
group element. Thus the length function $l(w)$ is the smallest $k$
such that $w=s_{i_1}s_{i_2}\ldots s_{i_k}$. The length function has
the following properties: $l(id)=0,\; l(s_i)=1,$ and
$l(s_iw)=l(w)\pm 1$ for each $w\in W$, $1\leq i\leq n$. Let $q$ be a
parameter which takes values in the interval $(0,1)$.
\begin{defn}
The \textit{Iwahori-Hecke algebra} $H$ corresponding to $W$ is the
vector space with the  basis $\left\{T_w| w\in W\right\}$ and the
multiplication given by
\begin{equation}
T_iT_w=\left\{
         \begin{array}{ll}
           T_{s_iw}, & \hbox{if}\;\; l(s_iw)=l(w)+1, \\
           (q-1)T_{w}+qT_{s_iw}, & \hbox{if}\;\; l(s_iw)=l(w)-1.
         \end{array}
       \right.
\end{equation}
where $T_i=T_{s_i}$ for $1\leq i<n$.
\end{defn}
The irreducible representations of the Iwahori-Hecke algebra $H$ are
in one-to-one correspondence with the irreducible representations of
the Coxeter group $W$. Let $\hat{W}$ be an index set  for the
irreducible representations of $W$, and for each $\lambda\in\hat{W}$
let $\chi_W^{\lambda}$ be the corresponding irreducible character of
$W$. If $\chi_H^{\lambda}$ is the character of the irreducible
representation of $H$ indexed by $\lambda\in\hat{W}$ then
$$
\chi_H^{\lambda}(T_w)\biggl|_{q=1}=\chi_{W}^{\lambda}(w)
$$
for all $w\in W$. In particular, the irreducible representations of
$W$ and $H$ indexed by the same $\lambda$, $\lambda\in\hat{W}$, have
the same dimensions. Define a trace $\vec{t}:\; H\rightarrow\C$ on
$H$ by
$$
\vec{t}(T_w)=\left\{
               \begin{array}{ll}
                 P_W(q), & \hbox{if}\;\;w=1,  \\
                 0, & \hbox{otherwise,}
               \end{array}
             \right.
$$
where $P_W(q)=\sum\limits_{w\in W}q^{l(w)}$ is the
Poincar$\acute{\mbox{e}}$ polynomial of the group $W$. The
\textit{generic degrees} are the constants $t_{\lambda}$ defined by
\begin{equation}\label{-11}
\vec{t}=\sum\limits_{\lambda\in\hat{W}}t_{\lambda}\chi^{\lambda}_H.
\end{equation}
Let $S(n)$ be the symmetric group. $S(n)$ is generated by the simple
transpositions $s_i=(i,i+1), 1\leq i\leq n-1$.  The irreducible
representation of $S(n)$, and of the corresponding Iwahori-Hecke
algebra $H$ are indexed by Young diagrams  with $n$ boxes. Let
$l(\lambda)$ be the number of nonzero rows in the Young diagram
$\lambda$ (the length of the Young diagram), and let $|\lambda|$ be
the number of boxes of $\lambda$. We number the rows and columns as
for matrices, and  denote by $\lambda_i$  and $\lambda_j'$ the
length of the $i^{\mbox{th}}$ row and of the $j^{\mbox{th}}$
respectively. The hook length $h(u)$ of a box $u$ in position
$(i,j)$ of $\lambda$ is
$$
h(u)=\lambda_i-i+\lambda_j'-j+1.
$$
Let $b(\lambda)=\sum\limits_{i=1}^{l(\lambda)}(i-1)\lambda_i$, and
introduce notations $[k]=1-q^k$, $[k]!=[1][2]\ldots [k]$. With these
notations the generic degrees $t_{\lambda}$ are given by
\begin{equation}\label{-12}
t_{\lambda}=\frac{q^{b(\lambda)}[n]!}{\prod_{u\in\lambda}[h(u)]}.
\end{equation}
For the Poincar$\acute{\mbox{e}}$ polynomial of $S(n)$ there is an
explicit formula
\begin{equation}\label{-13}
P_{S(n)}(q)=\prod\limits_{i=1}^{n-1}\frac{q^{i+1}-1}{q-1}=\frac{[n]!}{(1-q)^n}.
\end{equation}
\subsection{$q$-deformation of the Plancherel measures}\label{IwHSecII}
Consider equation (\ref{-11}) in the case of $W=S(n)$. When $w$ is
the unit element of $S(n)$  equation (\ref{-11}) takes the form
\begin{equation}\label{-14}
P_{S(n)}(q)=\sum\limits_{|\lambda|=n}t_{\lambda}\dim\lambda.
\end{equation}
where $\dim\lambda$ is the dimension of the irreducible
representation of $S(n)$ parameterized by $\lambda$, $|\lambda|=n$.
Alternatively, $\dim\lambda$ can be understood as the number of the
standard Young tableaux of the shape $\lambda$. A convenient
explicit formula for $\dim\lambda$ is
\begin{equation}
\dim\lambda=\frac{n!}{\prod_{u\in\lambda}h(u)},
\end{equation}
see, for example, Fulton and Harris \cite{fulton}, $\S 4.1$.
Inserting expressions for the generic degrees $t_{\lambda}$
(equation (\ref{-12})), and for the Poincar$\acute{\mbox{e}}$
polynomial $P_{S(n)}$ of $S(n)$ (equation (\ref{-13})) into formula
(\ref{-14}) we obtain the identity
\begin{equation}\label{-144}
\sum\limits_{|\lambda|=n}\frac{q^{b(\lambda)}\dim\lambda}{\prod_{u\in\lambda}[h(u)]}
=\frac{1}{(1-q)^n}.
\end{equation}
Denote by $\Y_n$ the set of  Young diagrams $\lambda$ with $n$
boxes, and introduce the following function of $\lambda$ on $\Y_n$
\begin{equation}\label{-15}
M^{(n)}_q(\lambda)=(1-q)^n\dim\lambda\frac{q^{b(\lambda)}}{\prod_{u\in\lambda}[h(u)]}.
\end{equation}
Then $\sum_{|\lambda|=n}M^{(n)}_q(\lambda)=1$, and each $M^{(n)}_q$
is a probability distribution on the set $\Y_n$ of Young diagrams
with $n$ boxes. We have
\begin{equation}
M^{(n)}_{q=1}(\lambda)=M^{(n)}_{Plancherel}(\lambda),\;\;\;\mbox{where}
\;\;
M^{(n)}_{Plancherel}(\lambda)=\frac{\left(\dim\lambda\right)^2}{n!}
\nonumber
\end{equation}
is the Plancherel measure. Thus $M^{(n)}_q$ can be understood as a
$q$-deformation of the Plancherel measure.

\subsection{Relation with non-uniform random permutations}\label{IwHSecIIIII}
It is a well known fact that the Plancherel measure is a push
forward of the uniform distribution on the symmetric group. In this
section we show that $M^{(n)}_q$ can be understood as a push forward
of a non-uniform distribution on the symmetric group $S(n)$.

 We assume that $S(n)$ is realized
as  the group of permutations of the set
$\left\{1,2,\ldots,n\right\}$. Let $\sigma$ be a permutation from
$S(n)$. We say that $i$, $i\in\left\{1,2,\ldots ,n\right\}$, is a
descent if $\sigma(i)>\sigma(i+1)$. Denote by $D(\sigma)$ the set of
all descents of $\sigma$, and define the major index of $\sigma$,
$\MAJ(\sigma)$, by the formula
$$
\MAJ(\sigma)=\sum\limits_{i\in D(\sigma)}i.
$$
Introduce a probability distribution on $S(n)$ by setting
\begin{equation}\label{01}
\mathbb{P}(\sigma)=\frac{q^{\MAJ(\sigma)}}{\sum\limits_{\sigma\in
S(n)}q^{\MAJ(\sigma)}}.
\end{equation}
If the value of the parameter $q$ approaches to 1 then $\mathbb{P}$
approaches to the uniform distribution on the symmetric group
$S(n)$, and if the value of the parameter $q$ approaches to 0 then
$\mathbb{P}$ approaches to the distribution concentrated at the unit
element $\sigma=e$ of the group $S(n)$.

Let $T$ be a standard Young tableau with entries $1,2,3,\ldots, n$.
Define a descent of $T$ to be an integer $i$ such that $i+1$ appears
in a row of $T$ lower than $i$, and define the descent set $D(T)$ to
be the set of all descents of $T$. For instance, the standard Young
tableau

\setlength{\unitlength}{4pt}
\begin{picture}(20,24)
\put(20,0){\framebox(4,4){$13$}}
\put(20,4){\framebox(4,4){$11$}} \put(24,4){\framebox(4,4){$12$}}
\put(20,8){\framebox(4,4){$9$}} \put(24,8){\framebox(4,4){$7$}}
\put(20,12){\framebox(4,4){$2$}} \put(24,12){\framebox(4,4){$5$}}
\put(28,12){\framebox(4,4){$8$}} \put(32,12){\framebox(4,4){$10$}}
 \put(20,16){\framebox(4,4){$1$}}
\put(24,16){\framebox(4,4){$3$}} \put(28,16){\framebox(4,4){$4$}}
\put(32,16){\framebox(4,4){$6$}}
\end{picture}
\\
has the descent set $\left\{1,4,6,8,10,12\right\}$. For any standard
Young tableau $T$ define the major index $\MAJ(T)$ by
$$
\MAJ(T)=\sum\limits_{i\in D(T)}i.
$$
\begin{prop}\label{P01}
For any Young diagram $\lambda$ we have
$$
\sum\limits_{T}q^{\MAJ(T)}=
\frac{q^{b(\lambda)}[n]!}{\prod_{u\in\lambda}[h(u)]},
$$
where $T$ ranges over all standard Young tableaux of shape
$\lambda$.
\end{prop}
\begin{proof}
The proof is given in Stanley \cite{stanley},  Chapter 7, pages
374-376.
\end{proof}
\begin{prop}\label{P02}
Let $\sigma\in S(n)$, and assume that $\sigma$  corresponds to the
pair $(P,Q)$ of standard Young tableaux of the same shape via the
Robinson-Schensted-Knuth algorithm. Then $D(P)=D(\sigma^{-1})$, and
$D(Q)=D(\sigma)$, where $D$ denotes the descent set.
\end{prop}
\begin{proof}
See Stanley \cite{stanley}, Chapter 7, page 382.
\end{proof}
\begin{prop}
The probability distribution $M^{(n)}_q$ defined by equation
(\ref{-15}) is a push forward of the non-uniform distribution
$\mathbb{P}$ on $S(n)$ (defined by equation (\ref{01})) via the
Robinson-Schensted-Knuth correspondence.
\end{prop}
\begin{proof}
Assume that $(P(\sigma), Q(\sigma))$ is the pair of standard Young
tableaux which is in one-to-one correspondence with $\sigma$ via the
Robinson-Schensted-Knuth algorithm, $\sigma$ is an element of
$S(n)$. Suppose that the probability of $\sigma$ is
$\mathbb{P}(\sigma)$, where $\mathbb{P}(\sigma)$ is given explicitly
by equation (\ref{01}). Then $\mathbb{P}(\sigma)$ equals  the
probability to find the pair $(P(\sigma), Q(\sigma))$ among all
possible pairs of Young diagrams with $n$ boxes, and of the same
shape. Denote this probability by $\mathbb{P}\left\{(P(\sigma),
Q(\sigma))\right\}$. Since $D(\sigma)=D(Q(\sigma))$ we have
$\MAJ(\sigma)=\MAJ(Q)$, and $\mathbb{P}\left\{(P(\sigma),
Q(\sigma))\right\}$ takes the form
$$
\mathbb{P}\left\{(P(\sigma),
Q(\sigma))\right\}=\frac{q^{\MAJ(Q)}}{\sum\limits_{\SP(P)=\SP(
Q)}q^{\MAJ(Q)}},
$$
where the sum is over all standard Young tableaux with $n$ boxes,
and of the same shape. Let us compute the probability of the event
that the tableaux in the pair $(P(\sigma),Q(\sigma))$ are of the
same shape $\lambda$, $|\lambda|=n$. This probability is
\begin{equation}\label{02}
\sum\limits_{\SP(P)=\SP(Q)=\lambda}\mathbb{P}\left\{(P(\sigma),
Q(\sigma))\right\}=\frac{\dim\lambda\sum\limits_{Q:
\SP(Q)=\lambda}q^{\MAJ(Q)}}{\sum\limits_{|\lambda|=n}\left(\dim\lambda\sum\limits_{Q:\SP(Q)
=\lambda}q^{\MAJ(Q)}\right)},
\end{equation}
where $\dim\lambda$ is the number of the standard Young tableaux of
the shape $\lambda$. The expression in the righthand side of
(\ref{02}) can be rewritten further using Proposition \ref{P01} and
formula (\ref{-144}). The result is
$$
\sum\limits_{\SP(P)=\SP(Q)=\lambda}\mathbb{P}\left\{(P(\sigma),
Q(\sigma))\right\}=(1-q)^n\dim\lambda\frac{q^{b(\lambda)}}{\prod_{u\in\lambda}[h(u)]}=M^{(n)}_q(\lambda).
$$
Therefore, $M^{(n)}_q(\lambda)$ is exactly the probability of the
event that the tableaux in the pair $(P(\sigma),Q(\sigma))$ are of
the same shape $\lambda$, where the pair $(P(\sigma),Q(\sigma))$
corresponds to permutation $\sigma$ via the Robinson-Schensted-Knuth
algorithm, and $\sigma$ is a random permutation from $S(n)$ with
respect to  the probability distribution $\mathbb{P}$ defined by
(\ref{01}). We conclude that $M^{(n)}_q$ is push forward of the
nonuniform distribution $\mathbb{P}$ on $S(n)$.

\end{proof}
\begin{rem}
1) Several $q$-analogs of the Plancherel measure were studied in a
paper by Fulman \cite{fulman} in connection with increasing and
decreasing subsequences in non-uniform random permutations. However
the measures considered in Ref. \cite{fulman} are different
from $M^{(n)}_q$.\\
2) The measure $M^{(n)}_q$ is a particular case of knot ergodic
central measures, see the book by Kerov \cite{kerov1}, Section 3,
$\S$4. The description of knot measures in the content of the
representation theory of the infinite-dimensional Hecke algebra
$H_{\infty}(q)$ can be found  in the paper by Vershik and Kerov
\cite{vershik1}.
\end{rem}
\section{Transition probabilities on the Young graph}
\subsection{The Young graph}
For two Young diagrams $\lambda$ and $\mu$ write
$\mu\nearrow\lambda$ (equivalently, $\lambda\searrow\mu$) if
$\mu\subset\lambda$ and $|\mu|=|\lambda|-1$, i.e. $\mu$ is obtained
from $\lambda$ by removing one box. Let $\Y$ denote the lattice of
Young diagrams ordered by inclusion. We consider $\Y$ as a graph
whose vertices are arbitrary Young diagrams $\mu$ and the edges are
couples $(\mu,\lambda)$ such that $\lambda\searrow\mu$. We shall
call $\Y$ the Young graph, and shall denote the level consisting of
the Young diagrams with $n$ boxes by $\Y_n$. In this content a
standard Young tableau can be understood as a  directed path
$$
\emptyset\nearrow\lambda^{(1)}\nearrow\ldots
\nearrow\lambda^{(n)}=\lambda
$$
exiting from the initial vertex $\lambda=\emptyset$ of the Young
graph. The dimension of  a Young diagram $\lambda$ is the number
$\dim\lambda$ defined recursively as follows: $\dim\emptyset=0$ for
the empty diagram $\lambda=\emptyset$, and
\begin{equation}\label{12}
\dim\Lm=\sum\limits_{\lm:\;\lm\nearrow\Lm}\dim\lm.
\end{equation}
It is clear from the definition above that $\dim\lambda$ is the
number of standard Young tableaux of the shape $\lambda$. Also note
that $\dim\lambda$ coincides with the dimension of the corresponding
representation of the symmetric group, and equation (\ref{12})
follows from the Young branching rule for the characters of the
finite symmetric group $S(n)$, $n=1,2,\ldots $.
\begin{rem}
The Young graph is a particular case of multiplicative graphs. Other
examples of multiplicative graphs  are the Jack graph, the Kingman
graph, the Schur graph, see the paper by Borodin and Olshanski
\cite{borodinolshanski1} for details and further references.
\end{rem}
\subsection{Harmonic functions on the Young graph}\label{S22}
A (real) valued function $\varphi(\lambda)$ is called a harmonic
function on the Young graph if it satisfies the condition
\begin{equation}\label{21}
\varphi(\lambda)=\sum\limits_{\Lambda: \Lm\searrow\lm}\varphi(\Lm)
\end{equation}
for any $\lambda\in\Y$. For the representation-theoretic meaning of
the harmonic functions on the Young graph see Refs.
\cite{vershik2,kerov2,borodinolshanski1}. We are interested in
nonnegative harmonic functions $\varphi$ normalized at the empty
diagram: $\varphi(\emptyset)=1$. As in Ref. \cite{borodinolshanski1}
we denote the set of such functions by $\mathcal{H}_1^{+}(\Y)$.
\begin{prop}\label{P21}
Let $\varphi\in\mathcal{H}_1^{+}(\Y)$, and let $M(\lambda)$ be a
function on the vertices of $\Y$ defined by
$M(\lambda)=\dim\lambda\varphi(\lambda)$. Denote by $M^{(n)}$ the
restriction of the function $M(\lambda)$ to the $n$th level $\Y_n$,
$n=0,1,2,\ldots$. Then $\sum_{|\lm|=n}M^{(n)}(\lm)=1$, i.e.
$M^{(n)}$ is a probability distribution on $\Y_n$.
\end{prop}
\begin{proof}
The proof is by induction. Since $\varphi(\emptyset)=1$,
$\dim(\emptyset)=1$ the claim is obviously valid for the level
$\Y_0$. Assume that the claim holds for the level $\Y_{n+1}$, i.e.
$$
\sum\limits_{|\Lm|=n+1}M^{(n+1)}(\Lm)=\sum\limits_{|\Lm|=n+1}\varphi(\Lm)\dim\Lm
=1.
$$
Insert the expression for $\dim\Lm$ (equation (\ref{12})) into the
formula written above, and obtain
$$
1=\sum\limits_{|\Lm|=n+1}\varphi(\Lm)\left(\sum\limits_{\lm:\;\lm\nearrow\Lm}\dim\lm\right)
=\sum\limits_{|\Lm|=n+1}\varphi(\Lm)\sum\limits_{|\lm|=n}\dim\lm.
$$
It is possible to  rewrite the right side  further as follows
$$
\sum\limits_{|\lambda|=n}\dim\lambda\sum\limits_{\Lambda:\Lm\searrow\lm}\varphi(\Lm).
$$
But the second sum above is precisely $\varphi(\lm)$, see equation
(\ref{21}), and the multiplication of the second sum on $\dim\lm$ is
$M^{(n)}(\lm)$. Thus  the formula $\sum_{|\lm|=n}M^{(n)}(\lm)=1$ is
obtained, and the claim of the proposition follows.
\end{proof}
\subsection{Transition and co-transition probabilities}\label{S23}
\begin{defn}
For two vertices $\lm$ and $\Lm$ of the Young graph $\Y$ such that
$\lm\in\Y_n$ and $\Lm\in\Y_{n+1}$ set
\begin{equation}{\label{22}}
q(\lm,\Lm)=\left\{
              \begin{array}{ll}
                \frac{\dim\lm}{\dim\Lm}, & \lm\nearrow\Lm, \\
                0, & \hbox{otherwise.}
              \end{array}
            \right.
\end{equation}
Then $\sum_{\lm\nearrow\Lm}q(\lm,\Lm)=1$, and we will refer to the
numbers $q(\lm,\Lm)$ as to the \textit{co-transition probabilities}
on the Young graph $\Y$.
\end{defn}
\begin{defn}
Assume that $\varphi(\lm)$ is a strictly positive valued harmonic
function, $\varphi\in\mathcal{H}_1^{+}(\Y)$, and set
\begin{equation}{\label{23}}
p(\lm,\Lm)=\left\{
              \begin{array}{ll}
                \frac{\varphi(\Lm)}{\varphi(\lm)}, & \Lm\searrow\lm, \\
                0, & \hbox{otherwise.}
              \end{array}
            \right.
\end{equation}
Then $\sum_{\Lm\searrow\lm}p(\lm,\Lm)=1$, and we refer to the
numbers $p(\lm,\Lm)$ as the \textit{transition probabilities} on the
Young graph $\Y$.
\end{defn}
Proposition \ref{P21} implies that the transition probabilities
define $M$ and $M^{(n)}$ uniquely.
\subsection{Transition probabilities for $M^{(n)}_{q}$}
A possible way to introduce transition probabilities on $\Y$ is to
use the Pieri rule for the Schur symmetric functions
\begin{equation}\label{3a1}
p_1\cdot s_{\lm}=\sum\limits_{\Lm\searrow\lm}s_{\Lambda},
\end{equation}
see Macdonald \cite{macdonald}, section I, $\S 5$.  Let
$\alpha=\{\alpha_i\}_{i=1}^{\infty}$,
$\beta=\{\beta_i\}_{i=1}^{\infty}$ be pairs of non-increasing
sequences of nonnegative  numbers satisfying the condition
\begin{equation}\label{3a2}
\sum\limits_{i=1}^{\infty}\alpha_i+\sum\limits_{i=1}^{\infty}\beta_i
\leq 1.
\end{equation}
Define the \textit{extended Schur functions}
$s_{\alpha}(\alpha,\beta)$ by the Frobenius formula
\begin{equation}\label{3a3}
s_{\lm}(\alpha,\beta)=\sum\limits_{|\rho|=n}\frac{1}{z_{\rho}}\chi^{\lambda}_{\rho}
p_{\rho}(\alpha,\beta),\;\; |\lambda|=n,
\end{equation}
where
$$
p_{\rho}(\alpha,\beta)=p_{\rho_1}(\alpha,\beta)\cdot
p_{\rho_2}(\alpha,\beta)\ldots,
$$
and the power sums $p_k(\alpha,\beta)$ are given by
\begin{equation}\label{3a4}
p_k(\alpha,\beta)=\left\{
                    \begin{array}{ll}
                      1, & k=1, \\
                      \sum\limits_{i=1}^{\infty}\alpha_i^k+(-1)^{k+1}\sum\limits_{i=1}^{\infty}\beta_i^k, &
k\geq 2.
                    \end{array}
                  \right.
\end{equation}
With this realization of the algebra $\Lambda$ of the symmetric
functions, condition (\ref{3a1}) implies that the ratios
\begin{equation}\label{3a5}
p(\lm,\Lm)=\left\{
             \begin{array}{ll}
               \frac{s_{\Lm}(\alpha,\beta)}{s_{\lm}(\alpha,\beta)}, & \Lm\searrow\lm, \\
               0, & \hbox{otherwise,}
             \end{array}
           \right.
\end{equation}
can be understood as transition probabilities.

Let $\alpha=\{(1-q)q^k\}_{k=0}^{\infty}, \beta=0$. In this case
\begin{equation}\label{3a6}
s_{\lm}(\alpha,\beta)=(1-q)^{|\lambda|}q^{b(\lambda)}\prod\limits_{b\in\lambda}[h(b)]^{-1},
\end{equation}
where $[k]=(1-q^k)$,
$b(\lambda)=\sum_{i=1}^{l(\lambda)}(i-1)\lambda_i$. Moreover,
$s_{\lm}(\alpha,\beta)$ can be understood as  harmonic functions on
the Young graph, as it follows from the Pieri rule (\ref{3a1}), and
from equation (\ref{3a4}). Harmonic functions determine uniquely
transition probabilities and distributions on the levels of the
Young graph, see sections \ref{S22}, \ref{S23}. In particular,
$s_{\lm}(\alpha,\beta)$ defined by equation (\ref{3a6}), and the
transition probabilities defined by equation (\ref{3a5}) lead to the
$q$-deformation of the Plancherel measure $M_q^{(n)}$,
 defined by equation (\ref{-15}). In this context, the
$q$-deformation of the Plancherel measure, $M_q^{(n)}$, is a Markov
probability measure on the Young graph $\Y$, with transition
probabilities defined by (\ref{3a5}) and (\ref{3a6}).
\begin{rem}
Other choices of the parameters $\alpha, \beta$ result in
probability distributions different from $M_q^{(n)}$, see Kerov
\cite{kerov1}, section 3.4.2, examples 1-5.
\end{rem}

\section{Continual diagrams and $q$-transition
distributions}\label{SECTIONCONTDIAGRAMS}
\subsection{Continual diagrams}
Continual diagrams were introduced by Kerov in Refs.
\cite{kerov0}-\cite{kerov03}, and  used further in Refs.
\cite{biane}, \cite{ivanov}. Here we recall the definition and some
properties of the continual diagrams.
\begin{defn}\label{DEF11}
A \textit{continual diagram} is a function $w(s)$ on $\R$ such that\\
(i) $|w(s_1)-w(s_2)|\leq |s_1-s_2|$ for any $s_1, s_2\in\R$ (the
Lipshitz condition). \\
(ii) There exists a point $s_0\in\R$, called the center of $w$, such
that $w(s)=|s-s_0|$ when $|s|$ is large enough.
\end{defn}
The set of all continual diagrams is denoted by $\D$, and the subset
of continual diagrams with the center 0 is denoted by $\D^{0}$.

To any $w\in\D$ assign a function
\begin{equation}
\sigma(s)=\frac{1}{2}\left(w(s)-|s|\right).
\end{equation}
This function is called the charge of the continual diagram $w$,
$w\in\D$.
\begin{prop}
a) $\sigma'(s)$ exists almost everywhere and satisfies
$$
|\sigma'(s)|\leq 1.
$$
b) $w(s)$ is uniquely determined by the second derivative
$\sigma''(s)$.\\
c) $\sigma'(s)$ is compactly supported, and
$$
\sigma'(s)=\left\{%
\begin{array}{ll}
    \left(w'(s)+1\right)/2\geq 0, & \hbox{for}\; s<0, \\
    \left(w'(s)-1\right)/2\leq 0, & \hbox{for}\; s>0. \\
\end{array}%
\right.
$$
\end{prop}
\begin{proof}
The first property follows from the Lipshitz condition (i) in
Definition \ref{DEF11}. The condition (ii) of Definition \ref{DEF11}
implies the second and the third properties of the function
$\sigma(s)$.
\end{proof}
\begin{defn} A continuous piecewise linear function $w:
\R\rightarrow\R$ is called a rectangular diagram if $w'(s)=\pm 1$
and there exists a constant $s_0$ such that $w(s)=|s-s_0|$ for
sufficiently large $|s|$.
\end{defn}
A rectangular diagram is completely determined by the coordinates of
its minima $\left\{x_k\right\}_{k=1}^{m+1}$ and those of its maxima
$\left\{y_k\right\}_{k=1}^m$. The  sequences
$\left\{x_k\right\}_{k=1}^{m+1}$ and $\left\{y_k\right\}_{k=1}^m$
interlace
$$
x_1<y_1<x_2<\ldots <x_{m}<y_{m}<x_{m+1}.
$$
Conversely, any pair of interlacing sequences uniquely determines a
rectangular diagram. The set of rectangular diagrams (or,
equivalently, the set of interlacing sequences) will be denoted by
$\D_R$.
\begin{example} \textit{Young diagrams.}\\
Given $\lambda\in\Y$ define a piecewise linear function $\lambda(s)$
with slopes $\pm 1$ and local minima and maxima  at two interlacing
sequences of integer points
$$ x_1<y_1<x_2<\ldots<x_m<y_m<x_{m+1},$$
where the $x_i$'s are the local minima, and the $y_i$'s are  the
local maxima of $\lambda(s)$, see Figure 1. The correspondence
$\lambda\rightarrow\lambda(s)$ gives an embedding
$$
\Y\rightarrow\D^{0},
$$
i.e. the set $\Y$ of Young diagrams is embedded into the subspace
$\D^0$ of continual diagrams with zero center.
\end{example}
\begin{example}
 \textit{Orthogonal polynomials.}
Let $\left\{P_m(x)\right\}_{m=0}^{\infty}$ be a sequence of
orthogonal polynomials defined with respect to a probability measure
$\mu$. The roots of two consecutive polynomials $P_{m+1}(x)$,
$P_{m}(x)$ interlace, so the roots of $P_{m+1}(x)$ can be understood
as minima, and the roots of $P_{m}(x)$ can be understood as maxima
of a rectangular diagram.
\end{example}

\subsection{$q$-deformations of $R$-functions}
Fix an interval $[a,b]$, where $a$ is strictly negative, and $b$ is
strictly positive.  Denote by $\D[a,b]$ the set of continual
diagrams with the property $w(s)=|s-s_0|$ for $s\notin [a,b]$. The
space $\D[a,b]$ is endowed with the uniform convergence topology.
Denote by $\D_R[a,b]$ the subspace of rectangular diagrams in
$\D[a,b]$. Note that the subspace of rectangular diagrams,
$\D_{R}[a,b]$, is dense in $\D[a,b]$.
 In addition, denote by $\M[a,b]$ the space
 of probability measures on the interval $[a,b]$.
\begin{defn}\label{DEFQRFUNCTIONS}
1) \textit{An $R$-function of a diagram }$w\in\D[a,b]$ is a function
$R_{w}(x)$ holomorphic outside the interval $[a,b]$, and defined by
\begin{equation}
R_{w}(x)=\frac{1}{x}\exp\left[-\int\limits_a^b\frac{d\sigma(s)}{s-x}\right]
=\frac{1}{x}\exp\left[-\frac{1}{2}\int\limits_a^b\frac{d\left(w(s)-|s|\right)}{s-x}\right].
\nonumber
\end{equation}
2) \textit{An $R$-function of a measure }$\mu\in\M[a,b]$ is a
function $R_{\mu}(x)$ holomorphic outside the interval $[a,b]$, and
defined by
\begin{displaymath}
R_{\mu}(x)= \int\limits_a^b\frac{\mu(ds)}{x-s}.
\end{displaymath}
\end{defn}
Definition \ref{DEFQRFUNCTIONS} is  due  to Kerov, see  Ref.
\cite{kerov0}, section 2.2. Now let us introduce natural
\textit{$q$-deformations} of the functions $R_{w}(x)$ and
$R_{\mu}(x)$.
\begin{defn}\label{DEFQDEFRFUNCTION}
Let $0<q\leq 1$, and assume that a real variable $x$ takes values
outside the interval $[a,b]$. For $0<q<1$ \textit{the
$q$-deformation of the $R$-function of a diagram} $w\in\D[a,b]$ is
defined by the expression
\begin{equation}
R_{w}(x;q)=\frac{1-q}{1-q^x}\exp\left[-\ln
q^{-1}\int\limits_a^b\frac{d\sigma(s)}{1-q^{x-s}}\right]
=\frac{1-q}{1-q^x}\exp\left[-\frac{1}{2}\ln
q^{-1}\int\limits_a^b\frac{d\left(w(s)-|s|\right)}{1-q^{x-s}}\right],
\nonumber
\end{equation}
and \textit{the $q$-deformation of the $R$-function of a measure}
$\mu\in\M[a,b]$ is defined by the expression
\begin{displaymath}
R_{\mu}(x;q)=(1-q)\int\limits_a^b\frac{\mu(ds)}{1-q^{x-s}}.
\end{displaymath}
 For $q=1$ the
$q$-deformation of the  $R$-function of a diagram $w\in\D[a,b]$ is
defined to be  $R_{w}(x)$, and the $q$-deformation of the
$R$-function of a measure $\mu\in\M[a,b]$ is defined to be
$R_{\mu}(x)$.
\end{defn}

\subsection{$q$-transition measures}
\begin{defn}\label{Defqcorrespondence} Fix $0<q\leq 1$. We call $\mu_q$,
$\mu_q\in\M[a,b]$, \textit{a $q$-transition measure of a continual
diagram} $w(.;q)$, $w(.;q)\in\D[a,b]$, if the functions
$R_{\mu_q}(x;q)$ and $R_{w(.;q)}(x;q)$ coincide.
\end{defn}
According to definition \ref{Defqcorrespondence}, if  $0<q<1$, and
$\mu_q$ is the $q$-transition measure of the diagram $w(.;q)$,
$w(.;q)\in\D[a,b]$, then
\begin{equation}\label{qcorrespondence}
\int\limits_a^b\frac{\mu_q(ds)}{1-q^{x-s}}=\frac{1}{1-q^x}\exp\left[-\frac{1}{2}\ln
q^{-1}\int\limits_a^b\frac{d\left(w(s;q)-|s|\right)}{1-q^{x-s}}\right],
\end{equation}
and if $q=1$ then the transition measure $\mu:=\mu_{q=1}$ of a
diagram $w(.):=w(.;q=1)$, and the diagram $w(.)$ are related by the
identity
\begin{equation}
\int\limits_a^b\frac{\mu(ds)}{x-s}=\frac{1}{x}\exp\left[-\frac{1}{2}\int\limits_a^b\frac{d\left(w(s)-|s|\right)}{s-x}\right].
\end{equation}
\begin{prop}\label{ProposRECT}
For a rectangular diagram $w$ with  the minima
$\left\{x_k\right\}_{k=1}^{m+1}$ and the maxima
$\left\{y_k\right\}_{k=1}^m$ the $q$-transition measure is the
probability measure supported by the finite set $\left\{x_1,\ldots
,x_{m+1}\right\}$ whose weights
$\left\{\mu_k(w;q)\right\}_{k=1}^{m+1}$ are given explicitly by the
formula
\begin{equation}\label{WEIGHTSQQ}
\mu_k(w;q)=\prod\limits_{i=1}^{k-1}\frac{1-q^{x_k-y_i}}{1-q^{x_k-x_i}}\prod\limits_{i=k+1}^{m+1}\frac{1-q^{x_k-y_{i-1}}}{1-q^{x_k-x_i}}.
\end{equation}
\end{prop}
\begin{proof}
Let $w$ be a rectangular diagram taken from $\D_R[a,b]$ with the
minima $\left\{x_k\right\}_{k=1}^{m+1}$ and the maxima
$\left\{y_k\right\}_{k=1}^m$. Then the function $R_w(x;q)$ takes the
form
\begin{equation}
R_{w}(x;q)=(1-q)\frac{\prod_{j=1}^m\left(1-q^{x-y_j}\right)}{\prod_{j=1}^{m+1}\left(1-q^{x-x_j}\right)}.
\nonumber
\end{equation}
Indeed, if $w$ is a rectangular diagram then the second derivative
of the function $\sigma(s)=\frac{1}{2}\left(w(s)-|s|\right)$ is
given by
$$
\sigma''(s)=\sum\limits_{k=1}^{m+1}\delta(s-x_k)-\sum\limits_{k=1}^m\delta(s-y_k)-\delta(s),
$$
and we can write
\begin{equation}\label{PROOO}
\begin{split}
\frac{\prod_{k=1}^m\left(1-q^{x-y_k}\right)}{\prod_{k=1}^{m+1}\left(1-q^{x-x_k}\right)}
&=\exp\left[-\int\limits_a^b\ln\left(1-q^{x-s}\right)\sigma''(s)ds-\ln\left(1-q^x\right)\right]
\\
&=\frac{1}{1-q^x}\exp\left[-\int\limits_a^b\ln\left(1-q^{x-s}\right)\sigma''(s)ds\right].
\end{split}
\end{equation}
The integration by parts shows that the righthand side of equation
(\ref{PROOO}) coincides with the function $R_{w}(x;q)$ divided by
$(1-q)$. The left-hand side of (\ref{PROOO}) can be rewritten as
\begin{equation}\label{PROOO1}
\frac{\prod_{k=1}^m\left(1-q^{x-y_k}\right)}{\prod_{k=1}^{m+1}\left(1-q^{x-x_k}\right)}
=\sum\limits_{k=1}^{m+1}\frac{\mu_k(w;q)}{1-q^{x-x_k}},
\end{equation}
where $\mu_k(w;q)>0$ for $k=1,\ldots ,m+1$;
$\sum\limits_{k=1}^{m+1}\mu_k(w;q)=1$; and the weights
$\left\{\mu_k(w;q)\right\}_{k=1}^{m+1}$ are given by
(\ref{WEIGHTSQQ}). Therefore, the equation
$R_{\mu}(x;q)=R_{w(.)}(x;q)$ implies in the case of rectangular
diagram that $\mu$ is supported by $\left\{x_k\right\}_{k=1}^{m+1}$,
and the weights $\left\{\mu_k(w;q)\right\}_{k=1}^{m+1}$ of $\mu$ are
given by (\ref{WEIGHTSQQ}).
\end{proof}
Recall that the $q$-deformation $M_q^{(n)}$ of the Plancherel
measure defined by equation (\ref{-15}) can be understood as a
Markov probability measure on the Young graph $\Y$, with the
transition probabilities $p(\lambda,\Lambda)$ defined by equations
(\ref{3a5}) and (\ref{3a6}). Consider the Young diagram $\lambda$ as
a rectangular diagram, see Figure 1. Denote by
$\left\{x_k\right\}_{k=1}^{m+1}$ the minima of $\lambda$, and by
$\{y_k\}_{k=1}^m$ the maxima of $\lambda$. Let us write
$\mu_k(\lambda;q)$ instead of $p(\lambda,\Lambda)$ if the square
that distinguishes $\Lambda$ from $\lambda$ is attached to the
minimum $x_k$ of $\lambda$.
\begin{prop}
The transition probabilities $\mu_k(\lambda;q)$ of the
$q$-deformation $M_q^{(n)}$ of the Plancherel measure are given by
formula (\ref{WEIGHTSQQ}), where in the left-hand side $w$ must be
replaced by $\lambda$. Thus, $\mu_k(\lambda;q)$ are the weights of
the $q$-transition measure of the diagram $\lambda$ in the sense of
definition \ref{Defqcorrespondence}.
\end{prop}
\begin{proof}
 See Kerov \cite{kerov1}, section 3.4.3.
\end{proof}
Formula (\ref{WEIGHTSQQ}) defines a bijection between the set
$\M^0[a,b]$ of probability measures on $[a,b]$ with finite support,
and the set of rectangular diagrams $\D_R[a,b]$. This bijection can
be extended by continuity to a homeomorphism of $\D[a,b]$ to
$\M[a,b]$.

\subsection{$q$-moments of continual diagrams}
Define the functions $p_1[w(.);q],p_2[w(.);q],\ldots $ on the space
of diagrams $\D[a,b]$ by setting
\begin{equation}\label{pmoments}
p_n[w(.);q]=1-\frac{n}{2}\ln
q^{-1}\int\limits_a^bq^{-ns}d\left(w(s)-|s|\right),
\end{equation}
 and define the functions $h_1[\mu;q],h_2[\mu;q],\ldots $ (where
 $\mu$ is a probability measure from $\M[a,b]$) by setting
\begin{equation}\label{Qhmoments}
h_n[\mu;q]=\int\limits_a^bq^{-ns}\mu(ds).
\end{equation}
We will refer to $h_1[\mu;q],h_2[\mu;q],\ldots $ as to
\textit{$q$-moments} of the probability measure $\mu$.
\begin{prop}\label{PROPOSITIONQQ}
 Fix a real parameter $q$ from the open interval $(0,1)$.
Assume that   a diagram $w(.;q)\in\D[a,b]$ and a probability measure
$\mu_q\in\M[a,b]$ are chosen in such a way that the relation
(\ref{qcorrespondence}) is satisfied for all $x$ outside the
interval $[a,b]$. Then the two sequences
$$
\biggl\{1-\frac{n}{2}\ln
q^{-1}\int\limits_a^bq^{-ns}d\left(w(s;q)-|s|\right)\biggr\}_{n=1,2,\ldots},
$$
and
$$
\biggl\{\int\limits_a^bq^{-ns}\mu_q(ds)\biggr\}_{n=1,2,\ldots}
$$
are related to each other in the same way as the systems of
generators of the algebra $\Lambda$ of the symmetric functions,
$\{\textbf{p}_n\}_{n=1,2,\ldots }$ and
$\{\textbf{h}_n\}_{n=1,2,\ldots}$. In other words, relation
(\ref{qcorrespondence}) is equivalent to
\begin{equation}\label{QQ1}
1+\sum\limits_{n=1}^{\infty}h_n[\mu_q;q]q^{nx}=\exp\left[\sum\limits_{n=1}^{\infty}p_n[w(.;q);q]q^{nx}\right].
\end{equation}
\end{prop}
\begin{proof}
Rewrite the righthand side of relation (\ref{qcorrespondence}) as
\begin{equation}
\begin{split}
&\frac{1}{1-q^x}\exp\left[-\ln
q^{-1}\int\limits_a^b\left(1-q^{x-s}\right)^{-1}\sigma'(s)ds\right] \\
&=\exp\left[\sum\limits_{n=0}^{\infty}\left[-\ln
q^{-1}\int\limits_a^bq^{-ns}\sigma'(s)ds\right]q^{nx}+\sum\limits_{n=1}^{\infty}\frac{q^{nx}}{n}\right]\\
&=\exp\left[\sum\limits_{n=1}^{\infty}\left(1-n\ln
q^{-1}\int\limits_a^bq^{-ns}\sigma'(s)ds\right)\frac{q^{nx}}{n}\right]\\
&=\exp\left[\sum\limits_{n=1}^{\infty}p_n[w(.;q);q]\frac{q^{nx}}{n}\right],
\end{split}
\nonumber
\end{equation}
where $\sigma(s)$ denotes the charge of the diagram $w_q$. (The fact
that $\int\limits_a^b\sigma'(s)ds=0$ was used to get the last
equation). Thus the righthand side of (\ref{qcorrespondence})
coincides with that of (\ref{QQ1}). The left-hand side of
(\ref{qcorrespondence}) can be rewritten as
$$
1+\sum\limits_{n=1}^{\infty}\int\limits_a^bq^{-ns}\mu_q(ds)q^{nx},
$$
which is $1+\sum\limits_{n=1}^{\infty}h_n[\mu_q;q]q^{nx}$. The
proposition is proved.
\end{proof}
\begin{cor}
If $\mu_q$ is the $q$-transition measure of the diagram $w(,;q)$,
and the parameter $q$ takes values in the open interval $(0,1)$ then
Proposition \ref{PROPOSITIONQQ} implies the relation
\begin{equation}
\begin{split}
R_{\mu_q}(x;q)&=(1-q)^{-1}\left(1+\sum\limits_{n=1}^{\infty}h_n[\mu_q;q]q^{nx}\right)\\
&=(1-q)^{-1}\exp\left[\sum\limits_{n=1}^{\infty}p_n[w(.;q);q]q^{nx}\right]=R_{w(.;q)}(x;q).
\end{split}
\nonumber
\end{equation}
\end{cor}
\subsection{$q$-deformation of the Markov-Krein correspondence}
Let $w\in\D[a,b]$, and define the function $F(s)$ by the formula
$$
F(s)=\frac{1}{2}\left(1+w'(s)\right)
$$
It is clear from definition \ref{DEF11} of continual diagrams that
$F(s)$ can be regarded as the distribution function of a signed
measure $\tau$. We will refer to the measure $\tau$ as to the
Rayleigh measure. Simple calculations show that the functions
$p_n[w(.);q]$ defined by equation (\ref{pmoments}) can be rewritten
as
\begin{equation}\label{momentstau}
p_n[w(.);q]=p_n[\tau;q]=\int\limits_a^bq^{-ns}\tau(ds)
\end{equation}
Therefore the functions $p_n[w(.);q]$  can be regarded as the
\textit{$q$-moments of the Rayleigh measure} $\tau$.
\begin{thm}\label{PRKREIN}
There is a relationship between a probability measure  $\mu_q$ on
$[a,b]$, and a Rayleigh measure $\tau_q$ on $[a,b]$ defined by the
identity
\begin{equation}\label{qKreinCorrespondence}
\int\limits_a^b\frac{\mu_q(ds)}{1-q^{x-s}}=\exp\left[\int\limits_a^b\ln\left(\frac{1}{1-q^{x-s}}\right)\tau_q(ds)\right].
\end{equation}
The probability measure $\mu_q$ and the Rayleigh measure $\tau_q$
determine each other uniquely via equation
(\ref{qKreinCorrespondence}).
\end{thm}
\begin{proof}
Equation (\ref{qKreinCorrespondence}) can be obtained from equation
(\ref{qcorrespondence}) with the integration by parts. To prove the
fact that $\tau_q$ and $\mu_q$ determine each other uniquely recall
that the moments
$$
h_n=h_n[\mu]=\int\limits_a^bs^n\mu(ds)
$$
 determine the finite measure $\mu$ uniquely. (This fact is known
 as the uniqueness of a solution for the Hausdorff Moment
 Problem). Using the obvious change of variables we can deduce
 that $h_n[\mu_q;q]$ defined by equation (\ref{Qhmoments})
 determine the probability measure $\mu_q$ uniquely. The moments $p_n[\tau;q]$
 defined by equation (\ref{momentstau}) also determine the Rayleigh
 measure $\tau_q$ uniquely. Furthermore, we have proved (see proposition
 \ref{PROPOSITIONQQ}) that equation (\ref{qcorrespondence}) is equivalent to
  the fact that the moments $h_n[\mu_q;q]$ and
 $p_n[\tau_q;q]$ are related to each other as the corresponding
 systems of generators of the algebra $\Lambda$ of symmetric
 functions. This implies that the moments $h_n[\mu_q;q]$ and
 $p_n[\tau_q;q]$ determine each other uniquely. The same arguments
 as in the proof of theorem 2.3 in Kerov \cite{kerov0}, section
 2.5, can be applied to complete the proof.
\end{proof}
\begin{thm}
Let $q$ be a fixed parameter which is taken from the interval
$(0;1]$. Then the relation $R_{\mu_q}(x;q)=R_{w(.;q)}(x;q)$ defines
the one-to-one correspondence between continual diagrams from
$\D[a,b]$, and the probability measures from $\M[a,b]$.
\end{thm}
\begin{proof}
The statement of the theorem in the case $q=1$ is proved in Kerov
\cite{kerov0,kerov03}. For $q\in(0,1)$ the statement of the theorem
follows immediately from theorem \ref{PRKREIN}, and from the known
fact that a diagram can be uniquely recovered from its Rayleigh
measure (see, for example, Kerov \cite{kerov03}).
\end{proof}

\section{Continual tableaux}
\subsection{Definition of continual tableaux}
\begin{defn}
The region $\D_w[a,b]=\left\{(s,v): |s|\leq v<w(s)\right\}$ is
called the subgraph of a continual diagram $w$, $w\in\D[a,b]$.
\end{defn}
\begin{defn}\label{DEFordering}
Let $w_1,w_2\in\D[a,b]$. We say that $w_1\prec w_2$ if the subgraph
of $w_1$ is a subset of the subgraph of $w_2$, i.e.
$\D_{w_1}[a,b]\subset\D_{w_2}[a,b]$.
\end{defn}
\begin{defn}
Let $t$ be a parameter which takes values in some interval
$[t_0,\infty)$. A continual tableau is a family of continual
diagrams from $\D[a,b]$, $w(.,t)$, which increases in $t$ (with
respect to the ordering introduced in definition \ref{DEFordering}).
\end{defn}

The function $\sigma(s,t)=\frac{1}{2}\left(w(s,t)-|s|\right)$ will
be referred to as the charge of a tableau $w(.,t)$.

\subsection{q-moments of continual tableaux}
\begin{prop}\label{PROPOSITIONtq}
Given a real number $q$ from the open interval $(0,1)$ assume that a
tableau $w(.,t;q)$ and a family $\mu_{t,q}$ of probability measures
from $\M[a,b]$ are related to each other by the formula
\begin{equation}\label{TQQ}
\int\limits_a^b\frac{\mu_{t,q}(ds)}{1-q^{x-s}}=\frac{1}{1-q^x}\exp\left[-\ln
q^{-1}\int\limits_a^b\frac{\partial\sigma(s,t;q)}{\partial
s}\left(1-q^{x-s}\right)^{-1}ds\right]
\end{equation}
for all $x$ outside the interval $[a,b]$. Then two sequences
$$
\biggl\{1-n\ln
q^{-1}\int\limits_a^bq^{-ns}\frac{\partial\sigma(s,t;q)}{\partial
s}\biggr\}_{n=1,2,\ldots},
$$
and
$$
\biggl\{\int\limits_a^bq^{-ns}\mu_{t,q}(ds)\biggr\}_{n=1,2,\ldots},
$$
are related to each other in the same way as the systems of
generators of the algebra $\Lambda$ of the symmetric functions,
$\{\textbf{p}_n\}_{n=1,2,\ldots }$ and
$\{\textbf{h}_n\}_{n=1,2,\ldots}$. In other words, relation
(\ref{TQQ}) is equivalent to
\begin{equation}\label{TQQ1}
1+\sum\limits_{n=1}^{\infty}h_n[\mu_{t,q};q]q^{nx}=\exp\left[\sum\limits_{n=1}^{\infty}p_n[w(.,t;q);q]q^{nx}\right],
\end{equation}
where the functions $\left\{h_n[\mu_{t,q};q]\right\}_{n=1}^{\infty}$
are defined by
\begin{equation}\label{momentyyyyyyyy}
h_n[\mu_{t,q};q]=\int\limits_a^bq^{-ns}\mu_{t,q}(ds),
\end{equation}
and the functions $\left\{p_n[w(.,t;q);q]\right\}_{n=1}^{\infty}$
are defined by
\begin{equation}\label{531}
p_n[w(.,t;q);q]=-n\ln
q^{-1}\int\limits_a^bq^{-ns}\;\frac{\partial\sigma(s,t;q)}{\partial
s}\;ds+1.
\end{equation}
\end{prop}
\begin{proof}The proof of this proposition is step by step
repetition of the proof of proposition \ref{PROPOSITIONQQ}.
\end{proof}
\begin{rem}
The equivalent form of equation (\ref{TQQ}) is
$$
R_{\mu_{t,q}}(x;q)=R_{w(.,t;q)}(x;q),
$$
where the functions $R_{\mu_{t,q}}(x;q)$ and $R_{w(.,t;q)}(x;q)$ are
defined by
\begin{equation}\label{RMU}
R_{\mu_{t,q}}(x;q)=(1-q)\int\limits_a^b\frac{\mu_{t,q}(ds)}{1-q^{x-s}},
\end{equation}
\begin{equation}\label{RWU}
R_{w(.,t;q)}(x;q)=\frac{1-q}{1-q^x}\exp\left[-\ln
q^{-1}\int\limits_a^b\frac{\partial\sigma(s,t;q)}{\partial
s}(1-q^{x-s})^{-1}ds\right]
\end{equation}
for all $x$ outside the interval $[a,b]$, and for all $q$ taking
values from the open interval $(0,1)$.
\end{rem}

\subsection{Dynamic equations}
\begin{thm}\label{theorem541}
The following dynamic equations are equivalent
\begin{equation}\label{i}
\int\limits_a^b\left(1-q^{x-s}\right)^{-1}\frac{\partial\sigma(s,t;q)}{\partial
t}ds=\left(1-q^x\right)^{-1}\exp\left[-\ln
q^{-1}\int\limits_a^b\left(1-q^{x-s}\right)^{-1}\frac{\partial\sigma(s,t;q)}{\partial
s}ds\right];
\end{equation}
\begin{equation}\label{ii}
\frac{\partial}{\partial t}p_n\left[w(.,t;q);q\right]
=n^2\ln^2q^{-1}\sum\limits_{|\lambda|=n}\prod\limits_{k=1}^{m(\lambda)}\frac{p_k^{r_k}[w(.,t;q);q]}{k^{r_k}r_k!},
\end{equation}
where $n=1,2,\ldots ,$ $\lambda=\left(1^{r_1},2^{r_2},\ldots
,m^{r_m}\right)$, and $m=m(\lambda)$;
\begin{equation}\label{iii}
\frac{\partial R_{w(.,t;q)}(x;q)}{\partial x}+\frac{1-q}{\ln
q^{-1}}R_{w(.,t;q)}^{-1}(x;q)\frac{\partial R_{w(.,t;q)}(x;q)
}{\partial t}=0.
\end{equation}
\end{thm}
\begin{proof}
Let us show that the first equation in the statement of the theorem,
equation (\ref{i}), implies equation (\ref{ii}).   Let $w(s,t;q)$ be
a tableau satisfying (\ref{i}), and let $\sigma(s,t;q)$ be the
charge of $w(.,t;q)$. Set
\begin{equation}\label{QMU}
\mu_{t,q}(ds)=\frac{\partial\sigma(s,t;q)}{\partial t}ds
\end{equation}
It is not hard to see that $\mu_{t,q}$ defined by equation
(\ref{QMU}) is a family of probability measures from $\M[a,b]$. If
the charge $\sigma(s,t;q)$ of $w(s,t;q)$ satisfies equation
(\ref{i}), then for every admissible $t$ the measure $\mu_{t,q}$ is
the $q$-transition measure of the diagram $w(s,t;q)$, see definition
\ref{DEFQRFUNCTIONS}. The moments $h_n[\mu_{t,q};q]$ of $\mu_{t,q}$
can be expressed as
\begin{equation}\label{2Z}
\begin{split}
h_n[\mu_{t,q};q]&=\int\limits_a^bq^{-ns}\frac{\partial\sigma(s,t;q)}{\partial
t}ds\\
&=\frac{\partial}{\partial t}\left(\int\limits_a^bq^{-ns}\sigma(s,t;q)ds\right)\\
&=\frac{1}{n^2\ln^2 q^{-1}}\frac{\partial}{\partial
t}p_n\left[w(.,t;q);q\right],
\end{split}
\end{equation}
where we have used the integration by parts to get the last equation
in (\ref{2Z}). If $\mu_{t,q}$ is defined by equation (\ref{QMU}),
then the first equation in the statement of the theorem coincides
with equation (\ref{TQQ}), and we can apply proposition
\ref{PROPOSITIONtq}. Namely, proposition \ref{PROPOSITIONtq} says
that the moments $\left\{h_n[\mu_{q,t};q]\right\}_{n=1,2,\ldots}$
and $\left\{p_n[w(.,t;q);q]\right\}_{n=1,2,\ldots}$ are related to
each other in the same way as the systems of the generators of the
algebra $\Lambda$ of the symmetric functions,
$\left\{\textbf{p}_n\right\}_{n=1,2,\ldots}$ and
$\left\{\textbf{h}_n\right\}_{n=1,2,\ldots}$. Therefore the
following relation holds
\begin{equation}\label{3Z}
h_n[\mu_{t,q};q]=\sum\limits_{|\lambda|=n}\prod\limits_{k=1}^{m(\lambda)}\frac{p_k^{r_k}[w(.,t;q);q]}{k^{r_k}r_k!},
\end{equation}
where $n=1,2,\ldots $, $\lambda=\left(1^{r_1},2^{r_2},\ldots,
m^{r_m}\right)$, and  $m=m(\lambda)$, see Macdonald
\cite{macdonald}, I, $\S 2$. From (\ref{2Z}) and (\ref{3Z}) we
obtain equation (\ref{ii}).

Let us show that the second equation in the statement of the theorem
implies equation (\ref{iii}). To this end define
$$
S(x,t;q)=\ln\left[\frac{R_{w(.,t;q)}(x;q)}{1-q}\right],
$$
where $R_{w(.,t;q)}(x;q)$ is given explicitly by equation
(\ref{RWU}). $S(x,t;q)$ can also be represented  as
$$
S(x,t;q)=\sum\limits_{n=1}^{\infty}\frac{p_n[w(.,t;q);q]q^{nx}}{n}.
$$
Differentiation of $S(x,t;q)$ with respect to the variable $t$ gives
\begin{equation}\label{543}
\frac{\partial S(x,t;q)}{\partial
t}=R^{-1}_{w(.,t;q)}(x;q)\frac{\partial R_{w(.,t;q)}(x;q)}{\partial
t}=\sum\limits_{n=1}^{\infty}\frac{q^{nx}}{n}\frac{\partial}{\partial
t}p_n[w(.,t;q);q].
\end{equation}
Observe that the first equation in the statement of the theorem is
equivalent to $R_{\mu_{t,q}}(x;q)=R_{w_{.,t;q}}(x;q)$ where
$\mu_{t,q}$ is defined by equation (\ref{QMU}). Note also that the
function $R_{\mu_{t,q}}(x;q)$ can be expanded in terms of the
moments $h_n[\mu_{t,q};q]$ as follows
$$
R_{\mu_{t,q}}(x;q)=(1-q)\sum\limits_{n=0}^{\infty}q^{nx}h_n[\mu_{t,q};q].
$$
Therefore the derivative of $R_{w(.,t;q)}(x;q)$ with respect to the
variable $x$ can be written as
\begin{equation}\label{544}
\begin{split}
\frac{\partial R_{w(.,t;q)}(x;q) }{\partial
x}&=(1-q)\sum\limits_{n=1}^{\infty}h_n[\mu_{t,q};q]\left(n\ln
q\right)q^{nx}\\
&=-\frac{(1-q)}{\ln
q^{-1}}\sum\limits_{n=1}^{\infty}\frac{\partial}{\partial
t}\left(p_n[w(.,t;q);q]\right)\frac{q^{nx}}{n},
\end{split}
\end{equation}
where we have used (\ref{2Z}). The comparison of (\ref{543}) and
(\ref{544}) gives the third equation in the statement of the
theorem, equation (\ref{iii}).
\end{proof}

\section{$q$-auto-model solutions}
\subsection{Definition of $q$-auto-models}
\begin{defn}\label{DEFA1}
Let $q$ be a fixed real number taken from the open interval $(0,1)$.
Assume that $w(s;q)$ (considered as a function of the variable $s$)
is an element of $\D[a,b]$.
 Assume further
that the subgraph $\D_w[a,b]$ of $w(s;q)$ is of unit area. A
continual tableau $w(s,t;q)$ defined in terms of $w(s;q)$ by
equation
\begin{equation}\label{AM}
w(s,t;q)=\sqrt{t}\;w(\frac{s}{\sqrt{t}};q^{\sqrt{t}}),\;\; t>0,
\end{equation}
is called a \textit{$q$-auto-model.}
\end{defn}
\subsection{A definition of the $q$-deformation of the limiting diagram}
\begin{defn}
Let $R_{\Omega(.;q)}(x;q)$ be the $q$-deformation of the
$R$-function of a continual diagram $\Omega(.;q)$, see definition
\ref{DEFQDEFRFUNCTION}. If $R_{\Omega(.;q)}$ satisfies the equation
\begin{equation}\label{OmegaOmega}
R_{\Omega(.;q)}\left(1-q^{x-\frac{\ln
q^{-1}}{1-q}R_{\Omega(.;q)}}\right)=1-q,
\end{equation}
then $\Omega(.;q)$ is referred to as the $q$-deformation of the
limiting diagram $\Omega(s)$ defined by equation
(\ref{OOOOMMMMEEEGA}).
\end{defn}
\begin{rem}
If $q$ in equation (\ref{OmegaOmega}) approaches 1, then
(\ref{OmegaOmega}) is reduced to equation
$R_{\Omega}(x-R_{\Omega})=1$. The solution of this equation
vanishing at $x\rightarrow +\infty$ is the $R$-function of the
diagram $\Omega(s)$ defined by equation (\ref{OOOOMMMMEEEGA}).
\end{rem}
\subsection{The $q$-auto-model solution as the $q$-deformation of the limiting diagram}
\begin{thm}\label{THEOREM5222}
Let $w(s;q)$ be an arbitrary diagram of unit area, and
$w(s,t;q)=\sqrt{t}w\left(\frac{s}{\sqrt{t}}; q^{\sqrt{t}}\right)$ be
the corresponding $q$-auto-model. If the charge $\sigma(s,t;q)$ of
$w(s,t;q)$ satisfies equation (\ref{i}), then $w(s;q)=\Omega(s;q)$.
\end{thm}
\begin{proof}
It is easy to check that the moments $p_1[w(.,t;q);q],
p_2[w(.,t;q);q], \ldots$ of the $q$-auto-model $w(s,t;q)$  coincide
with the moments $p_1[w(.;q^{\sqrt{t}});q^{\sqrt{t}}],
p_2[w(.;q^{\sqrt{t}});q^{\sqrt{t}}], \ldots$ of the diagram
$w(\frac{s}{\sqrt{t}},q^{\sqrt{t}})$:
\begin{equation}
p_n[w(.,t;q);q]=p_n[w(.;q^{\sqrt{t}});q^{\sqrt{t}}]. \nonumber
\end{equation}
Indeed, we have
\begin{equation}
\begin{split}
p_n[w(.,t;q);q]&=-n\ln
q^{-1}\int\limits_a^bq^{-ns}\frac{\partial}{\partial
s}\sigma(s,t;q)ds+1\\
&=-n\ln q^{-1}\int\limits_a^bq^{-ns}\frac{\partial}{\partial
s}\left[\frac{1}{2}\left(w(s,t;q)-|s|\right)\right]ds+1\\
&=-n\ln q^{-1}\int\limits_a^bq^{-ns}\frac{\partial}{\partial
s}\left[\frac{1}{2}\left(\sqrt{t}\cdot
w(\frac{s}{\sqrt{t}};q^{\sqrt{t}})-\sqrt{t}\cdot|\frac{s}{\sqrt{t}}|\right)\right]ds+1\\
&=-n\ln\left[q^{\sqrt{t}}\right]^{-1}\int\limits_{a/\sqrt{t}}^{b/\sqrt{t}}\left[q^{\sqrt{t}}\right]^{-nu}\frac{\partial}{\partial
u}\left[\frac{1}{2}\left(
w(u;q^{\sqrt{t}})-|u|\right)\right]du+1\\
&=p_n[w(.;q^{\sqrt{t}});q^{\sqrt{t}}].
\end{split}
\nonumber
\end{equation}
This enables us to express the function $R_{w(.,t;q)}(x;q)$ which
corresponds to the $q$-auto-model $w(.,t;q)$
 in terms of the function
$R_{w(.;q^{\sqrt{t}})}(x;q^{\sqrt{t}})$ which corresponds to the
diagram $w(u,q^{\sqrt{t}})$:
\begin{equation}\label{ZZZ}
\begin{split}
R_{w(.,t;q)}(x;q)&=(1-q)\exp\left[\sum\limits_{n=1}^{\infty}\frac{p_n[w(.,t;q);q]q^{nx}}{n}\right]\\
&=(1-q)\exp\left[\sum\limits_{n=1}^{\infty}\frac{p_n[w(.;q^{\sqrt{t}});q^{\sqrt{t}}]q^{n\left(\frac{x}{\sqrt{t}}\right)\sqrt{t}}}{n}\right]\\
&=\frac{1-q}{1-q^{\sqrt{t}}}\;R_{w(.;q^{\sqrt{t}})}(\frac{x}{\sqrt{t}};q^{\sqrt{t}}).
\end{split}
\end{equation}
Introduce new variables
$$
u=\frac{x}{\sqrt{t}},\;\; Q=q^{\sqrt{t}}.
$$
By equation (\ref{ZZZ}) we have
\begin{equation}
R_{w(.,t;q)}(x;q)=\frac{1-q}{1-Q}R(u;Q), \;\mbox{where}\;\;
R(u;Q)=R_{w(.;q^{\sqrt{t}})}(\frac{x}{\sqrt{t}};q^{\sqrt{t}}).
\nonumber
\end{equation}
The differential equation for the function $R_{w(.,t;q)}(x;q)$ (the
third equation in theorem \ref{theorem541}) leads to the following
partial differential equation for the function $R(u;Q)$:
\begin{equation}
\frac{\partial}{\partial u}\left[R^2(u;Q)\right]-\frac{1-Q}{\ln
Q^{-1}}u\frac{\partial R(u;Q)}{\partial
u}-QR(u;Q)-Q(1-Q)\frac{\partial R(u;Q)}{\partial Q}=0. \nonumber
\end{equation}
Set $R(u;Q)=\frac{1-Q}{\ln Q^{-1}}\mathrm{r}(u;Q)$, and introduce a
real parameter $\varrho$, $\varrho>0$, by the relation
$Q=\exp\left(-\varrho\right)$.  Then $\mathrm{r}(u;\varrho)$
satisfies the following partial quasi-linear differential equation
\begin{equation}\label{629}
2\mathrm{r}\frac{\partial}{\partial
u}\mathrm{r}-u\frac{\partial}{\partial
u}\mathrm{r}+\varrho\frac{\partial}{\partial\varrho}\mathrm{r}=\mathrm{r}.
\end{equation}
 Equation (\ref{629}) is a quasi-linear
partial differential equation in two variables, and can be solved by
the method of characteristics. Namely, for the partial differential
equation (\ref{629}) the characteristic equations are:
\begin{equation}\label{6211}
\frac{du}{ds}=2\mathrm{r}-u,\;\;\frac{d\varrho}{ds}=\varrho,\;\;\frac{d\mathrm{r}}{ds}=\mathrm{r}.
\end{equation}
Equations (\ref{6211}) can be rearranged to two ordinary
differential equations:
\begin{equation}\label{6212}
\frac{d\mathrm{r}}{\mathrm{r}}=\frac{d\varrho}{\varrho},
\end{equation}
\begin{equation}\label{6213}
\frac{du}{2\mathrm{r}-u}=\frac{d\varrho}{\varrho}.
\end{equation}
The integration of the first equation above gives
$\mathrm{r}=c_1\varrho$. Inserting this into (\ref{6213}) we obtain:
\begin{equation}
\frac{du}{2c_1\varrho-u}=\frac{d\varrho}{\varrho},
\;\;\mbox{or}\;\;\frac{du}{d\varrho}=2c_1-\frac{u}{\varrho}.
 \nonumber
\end{equation}
Integrating the last equation we find
$$
u=c_1\varrho+\frac{c_2}{\varrho}.
$$
Our first integrals, therefore, are
$c_1=f(u,\varrho,\mathrm{r})=\frac{\mathrm{r}}{\varrho}$ and
$c_2=g(u,\varrho,\mathrm{r})=\varrho(u-\mathrm{r})$. The general
solution is found by setting $f=F(g)$, which leads to the relation
\begin{equation}\label{rgeneral}
 \frac{\mathrm{r}}{\varrho}=F(\varrho(u-\mathrm{r})),
\end{equation}
where $F$ is an arbitrary function. Observe that
$$
\mathrm{r}(u;\varrho)=
\frac{\varrho}{1-\exp(-\varrho)}R_{w(.;e^{-\varrho})}(u;e^{-\varrho})=\frac{\varrho}{1-\exp(-\varrho)}R_{\mu_{e^{-\varrho}}}(u;e^{-\varrho}),
$$
where $\mu_{e^{-\varrho}}$ is the $q$-transition measure of the
diagram $w(.;e^{-\varrho})$. It follows that
$$
\mathrm{r}(u;\varrho)=\varrho\int\limits_{a}^{b}\frac{\mu_{e^{-\varrho}}(ds)}{1-e^{-\varrho(u-s)}},
$$
and from this equation we conclude that $\mathrm{r}(u,\varrho)$
approaches to $\varrho(1-e^{-\varrho u})^{-1}$ as
$u\rightarrow\infty$. This enables us to determine  the function $F$
in (\ref{rgeneral}) explicitly:
$$
F(x)=\frac{1}{1-\exp{(- x)}}.
$$
Consequently, the function $\mathrm{r}(u,\varrho)$ satisfies the
equation
$$
\mathrm{r}=\varrho\left(1-e^{-\varrho(u-\mathrm{r})}\right)^{-1}.
$$
If we rewrite this equation in terms of
$R_{w(.,e^{-\varrho})}(u;e^{-\varrho})$, and replace $\varrho$ by
$\ln q^{-1}$, and $u$ by $x$, we obtain
\begin{equation}
R_{w(.;q)}\left(1-q^{x-\frac{\ln q^{-1}}{1-q}R_{w(.;q)}}\right)=1-q.
\nonumber
\end{equation}
Therefore the functions $R_{w(.;q)}(x;q)$ and $R_{\Omega(.;q)}(x;q)$
coincide for all admissible values of $x$ and $q$. This implies
$w(s;q)=\Omega(s;q)$.
\end{proof}

\section{The asymptotics of the general solution}
\subsection{The large $t$ asymptotics of the functions
$p_n[w(.,t;q)]$ and $h_n[w(.,t;q)]$}
 Assume that the charge
$\sigma(s,t;q)$ of a diagram $w(s,t;q)$ satisfies equation (\ref{i})
of theorem \ref{theorem541}. Then the $q$-transition measure
$\mu_{t,q}$ of the diagram $w(s,t;q)$ is
$$
\mu_{t,q}(ds)=\frac{\partial\sigma(s,t;q)}{\partial t}ds.
$$
Let $\left\{h_n[\mu_{t,q};q]\right\}_{n=1}^{\infty}$ be the
$q$-moments of $\mu_{t,q}$ (see equation (\ref{momentyyyyyyyy})),
and for every $n=1,2,\ldots $ set
$h_n[w(.,t;q);q]:=h_n[\mu_{t,q};q]$. Recall that the functions
$p_n[w(.,t;q);q]$ are defined by equation (\ref{531}).
\begin{lem}\label{lemma711}
There exist functions $\{\check{p}_n[q]\}_{n=1}^{\infty}$ and
$\{\check{h}_n[q]\}_{n=1}^{\infty}$, which are independent on $t$,
such that
\begin{equation}\label{713}
p_n\left[w(.,t;q^{\frac{1}{\sqrt{t}}});
q^{\frac{1}{\sqrt{t}}}\right]=\check{p}_n[q]+o(t^{-1/2}),
\end{equation}
\begin{equation}\label{714}
h_n\left[w(.,t;q^{\frac{1}{\sqrt{t}}});
q^{\frac{1}{\sqrt{t}}}\right]=\check{h}_n[q]+o(t^{-1/2}),
\end{equation}
as $t\rightarrow\infty$.
\end{lem}
\begin{proof}
Comparing the righthand sides of equations (\ref{2Z}) and (\ref{3Z})
we obtain the following system of differential equations
\begin{equation}\label{715}
\frac{\partial}{\partial t}p_n\left[w(.,t;q);q\right]=n^2\ln^2
q^{-1}
\left\{\sum\limits_{|\lambda|=n}\prod\limits_{k=1}^{m(\lambda)}\frac{p_k^{r_k}[w(.,t;q);q]}{k^{r_k}r_k!}\right\},
\end{equation}
where $n=1,2,\ldots $, $\lambda=\left(1^{r_1},2^{r_2},\ldots,
m^{r_m}\right)$, and  $m=m(\lambda)$. Setting
$$
\varsigma
:=t\ln^2q^{-1},\;\;\mbox{and}\;\;y_n(\varsigma):=p_n[w(.,t;q);q],
$$
we obtain differential equations for functions
$\{y_n(\varsigma)\}_{n=1}^{\infty}$
\begin{equation}\label{SHEST}
\frac{dy_n(\varsigma)}{d\varsigma}
=n^2\left\{\sum\limits_{|\lambda|=n}\prod\limits_{k=1}^{m(\lambda)}\frac{y_k^{r_k}(\varsigma)}{k^{r_k}r_k!}\right\},\;\;
n=1,2,\ldots .
\end{equation}
The first equations of the system above are
$$
\dot{y}_1=y_1
$$
$$
\dot{y}_2=2y_1^2+2y_2
$$
$$
\dot{y}_3=\frac{3}{2}y_1^3+\frac{9}{2}y_2y_1+3y_3
$$
$$
\dot{y}_4=\frac{2}{3}y_1^4+4y_2y_1^2+\frac{16}{3}y_3y_1+2y_2^2+4y_4
$$
$$
\dot{y}_5=\frac{5}{24}y_1^5+\frac{25}{12}y_2y_1^3+\frac{25}{6}y_3y_1^2+\frac{25}{8}y_1y_2^2+\frac{25}{4}y_4y_1+\frac{25}{6}y_2y_3+5y_5
$$
$$
\vdots
$$
Successively solving these equations we find
$$
y_1(\varsigma)=y_1(0)e^{\varsigma}
$$
$$
y_2(\varsigma)=[y_2(0)+2y_1^2(0)\varsigma]e^{2\varsigma}
$$
$$
y_3(\varsigma)=\left[y_3(0)+\frac{3}{2}y_1(0)\left[3y_2(0)+y_1^2(0)\right]\varsigma+\frac{9}{2}y_1^3(0)\varsigma^2\right]e^{3\varsigma}
$$
\begin{equation}
\begin{split}
y_4(\varsigma)=\biggl[y_4(0)+\left[\frac{2}{3}y_1^4(0)+4y_1^2(0)y_2(0)+\frac{16}{3}y_1(0)y_3(0)+2y_2^2(0)\right]\varsigma\\
+\left[16y_1^2(0)y_2(0)+8y_1^4(0)\right]\varsigma^2+\frac{32}{3}y_1^4(0)\varsigma^3
\biggr]e^{4\varsigma}
\end{split}
\nonumber
\end{equation}
$$
\vdots
$$
Generally, $y_n(\varsigma)$ is a polynomial in $\varsigma$ of degree
$n-1$ multiplied by $e^{n\varsigma}$, and the coefficients of this
polynomial are homogeneous.

Returning to the functions
$\left\{p_n\left[w(.,t;q);q\right]\right\}_{n=1,2,\ldots}$ we obtain
$$
p_1[w(.,t;q);q]=p_1[w(.,t=0;q);q]\;e^{t\ln^2q^{-1}}
$$
$$
p_2[w(.,t;q);q]=\left[p_2[w(.,t=0;q);q]+2p_1^2[w(.,t=0;q);q](t\ln^2q^{-1})\right]\;e^{2t\ln^2q^{-1}}
$$
\begin{equation}
\begin{split}
p_3[w(.,t;q);q]&=\biggl[p_3[w(.,t=0;q);q]\\
&+\frac{3}{2}p_1[w(.,t=0;q);q]\left[3p_2[w(.,t=0;q);q]+p_1^2[w(.,t=0;q);q]\right](t\ln^2q^{-1})\\
&+\frac{9}{2}p_1^3[w(.,t=0;q);q](t\ln^2q^{-1})^2\biggr]\;e^{3t\ln^2q^{-1}}
\end{split}
\nonumber
\end{equation}

\begin{equation}
\begin{split}
p_4[w(.,t;q);q]&=\biggl[p_4[w(.,t=0;q);q]\\
&+\biggl[\frac{2}{3}p_1^4[w(.,t=0;q);q]+4p_1^2[w(.,t=0;q);q]p_2[w(.,t=0;q);q]\\
&+\frac{16}{3}p_1[w(.,t=0;q);q]p_3[w(.,t=0;q);q]+2p^2_2[w(.,t=0;q);q]\biggr](t\ln^2q^{-1})\\
&+\left[16p_1^2[w(.,t=0;q);q]p_2[w(.,t=0;q);q]+8p_1^4[w(.,t=0;q);q]\right](t\ln^2q^{-1})^2\\
&+
\frac{32}{3}p_1^4[w(.,t=0;q);q](t\ln^2q^{-1})^3\biggr]\;e^{4t\ln^2q^{-1}}
\end{split}
\nonumber
\end{equation}
$$
\vdots
$$
Now it is clear that if
\begin{equation}\label{716}
p_n[w(.,t=0;q^{\frac{1}{\sqrt{t}}});q^{\frac{1}{\sqrt{t}}}]=1+o\left(\frac{1}{\sqrt{t}}\right)
\end{equation}
as $t\rightarrow\infty$, then (\ref{713}) holds. But (\ref{716})
follows immediately from (\ref{531}). Since the relation between
functions $\{p_n[w(.,t;q)]\}_{n=1}^{\infty}$ and
$\{h_n[w(.,t;q)]\}_{n=1}^{\infty}$ is homogeneous, equation
(\ref{714}) holds as well.
\end{proof}
\begin{cor}\label{CORMOMENTS}
Let $\{y_n(\varsigma)\}_{n=1,2,\ldots}$ be the solution of the
system of differential equations given by equation (\ref{SHEST})
which satisfies the initial conditions $y_n(0)=1,\;n=1,2,\ldots$
Then the limiting values $\check{p}_n[q]$ in equation (\ref{713})
are given by
$$
\check{p}_n[q]=y_n[\ln^2q],\;\;n=1,2,\ldots,
$$
 and the limiting values $\check{h}_n[q]$ in equation
(\ref{714}) are given by
$$
\check{h}_n[q]=\left\{\sum\limits_{|\lambda|=n}\prod\limits_{k=1}^{m(\lambda)}\frac{y_k^{r_k}[\ln^2q]}{k^{r_k}r_k!}\right\}.
$$
where $n=1,2,\ldots $, $\lambda=\left(1^{r_1},2^{r_2},\ldots,
m^{r_m}\right)$, and  $m=m(\lambda)$.
\end{cor}
\begin{proof}
In order to obtain the values of $\check{p}_n[q]$ we need to compute
the limits
$\underset{t\rightarrow\infty}{\lim}p_n[w(.,t;q^{\frac{1}{\sqrt{t}}});
q^{\frac{1}{\sqrt{t}}}]$. Since
$\underset{t\rightarrow\infty}{\lim}p_n[w(.,t=0;q^{\frac{1}{\sqrt{t}}});
q^{\frac{1}{\sqrt{t}}}]=1$, $n=1,2,\ldots$ it is not hard to
conclude from the proof of the lemma above that
$\underset{t\rightarrow\infty}{\lim}p_n[w(.,t;q^{\frac{1}{\sqrt{t}}});
q^{\frac{1}{\sqrt{t}}}]=y_n(\ln^2q)$, $n=1,2,\ldots$, where
$\{y_n(\varsigma)\}_{n=1,2,\ldots}$ is the solution of the system of
differential equations  (\ref{SHEST}) which satisfies the initial
conditions $y_n(0)=1,\;n=1,2,\ldots$.
\end{proof}

\subsection{The common asymptotics of solutions}
\begin{thm}
Assume that the charge  $\sigma(s,t;q)$ of a tableau $w(s,t;q)$
satisfies (\ref{i}). Then
\begin{equation}
\underset{t\rightarrow\infty}{\lim}\frac{1}{\sqrt{t}}w(s\sqrt{t},t;q^{\frac{1}{\sqrt{t}}})=\Omega(s;q)
\nonumber
\end{equation}
uniformly in $s$ and $q$.
\end{thm}
\begin{proof}
 Define the normalized tableau
\begin{equation}\label{WTAB}
W(s,t;q)=\frac{1}{\sqrt{t}}\;w\left(s\sqrt{t},t;q^{\frac{1}{\sqrt{t}}}\right),\;\;
t>0.
\end{equation}
If the tableau $w(s,t;q)$ is a family of continual diagrams from
$\D[a,b]$ then the normalized tableau $W(s,t;q)$  is the family of
continual diagrams from $\D[a\sqrt{t},b\sqrt{t}]$. The functions
$p_n[W(s,t;q);q]$  can be expressed as
\begin{equation}
p_n[W(.,t;q);q]=-n\ln
q^{-1}\int\limits_{a\sqrt{t}}^{b\sqrt{t}}q^{-ns}\frac{\partial\Xi(s,t;q)}{\partial
s}ds+1, \nonumber
\end{equation}
where $\Xi(s,t;q)$ is the charge of the normalized diagram
$W(s,t;q)$. Let us express $\Xi(s,t;q)$ in terms of the charge
$\sigma(s,t;q)$ of the initial diagram $w(s,t;q)$:
\begin{equation}
\begin{split}
\Xi(s,t;q)&=\frac{1}{2}\left(W(s,t;q)-|s|\right)\\
&=\frac{1}{2}\left(\frac{1}{\sqrt{t}}\;w(s\sqrt{t},t;q^{\frac{1}{\sqrt{t}}})-\frac{1}{\sqrt{t}}|s\sqrt{t}|\right)\\
&=\frac{1}{\sqrt{t}}\;\sigma(s\sqrt{t},t;q^{\frac{1}{\sqrt{t}}}).
\end{split}
\nonumber
\end{equation}
Inserting this  into the integral for $p_n[W(.,t;q);q]$, and
changing the variables of the integration we obtain
\begin{equation}\label{RR}
p_n[W(.,t;q);q]=p_n[w(.,t;q^{\frac{1}{\sqrt{t}}});q^{\frac{1}{\sqrt{t}}}].
\end{equation}
By lemma \ref{lemma711} this implies the large $t$ asymptotic
relation
\begin{equation}
p_n[W(.,t;q);q]=\check{p}_n+o(t^{-1/2}),
\end{equation}
where $\check{p}_n$ are independent on $t$. Let the functions
$h_n[W(.,t;q);q]$ be defined in terms of the functions
$p_n[W(.,t;q);q]$ by the formula
\begin{equation}\label{HHP}
h_n[W(.,t;q);q]=\sum\limits_{|\lambda|=n}\prod\limits_{k=1}^{m(\lambda)}\frac{p_k^{r_k}[W(.,t;q);q]}{k^{r_k}r_k!},
\end{equation}
where $n=1,2,\ldots $, $\lambda=\left(1^{r_1},2^{r_2},\ldots,
m^{r_m}\right)$, and  $m=m(\lambda)$. From equation (\ref{HHP}) we
obtain the large $t$ asymptotic relation  for the functions
$h_n[W(.,t;q);q]$
\begin{equation}\label{Has}
h_n[W(.,t;q);q]=\check{h}_n+o(t^{-1/2}),
\end{equation}
where $\check{h}_n$ are independent on $t$. The sequences
$\left\{p_n[W(.,t;q);q]\right\}_{n=1,2,\ldots}$ and
$\left\{h_n[W(.,t;q);q]\right\}_{n=1,2,\ldots}$ are related with
each other as the sequences of the corresponding generators of the
algebra $\Lambda$ of symmetric functions,
$\left\{\textbf{p}_n\right\}_{n=1,2,\ldots}$ and
$\left\{\textbf{h}_n\right\}_{n=1,2,\ldots}$. Therefore the function
$R_{W(.,t;q)}(x;q)$ can be represented in two ways:
$$
R_{W(.,t;q)}(x;q)=(1-q)\;\exp\left[\sum\limits_{n=1}^{\infty}\frac{p_n[W(.,t;q);q]q^{nx}}{n}\right],
$$
and
$$
R_{W(.,t;q)}(x;q)=(1-q)\;\left(1+\sum\limits_{n=1}^{\infty}h_n[W(.,t;q);q]q^{nx}\right).
$$
>From the equation just written above, and from asymptotic relation
(\ref{Has}) we  conclude that
$$
\underset{t\rightarrow\infty}{\lim}R_{W(.,t;q)}(x;q)=\check{R}(x;q),\;\;\mbox{and}\;\;
\underset{t\rightarrow\infty}{\lim}\left[t\;\frac{\partial
R_{W(.,t;q)}(x;q)}{\partial t}\right]=0,
$$
where $\check{R}(x;q)$ is defined in terms of $\check{h}_n$ by
\begin{equation}\label{CH}
\check{R}(x;q)=(1-q)\;\left(1+\sum\limits_{n=1}^{\infty}\check{h}_nq^{nx}\right).
\end{equation}
Equation (\ref{RR}) also implies the relation
\begin{equation}\label{R612}
R_{w(.,t;q)}(x;q)=\frac{1-q}{1-q^{\sqrt{t}}}\;R_{W(.,t;q^{\sqrt{t}})}(\frac{x}{\sqrt{t}};q^{\sqrt{t}}).
\end{equation}
Indeed,
\begin{equation}
\begin{split}
R_{W(.,t;q)}(x;q)&=(1-q)\;\exp\left[\sum\limits_{n=1}^{\infty}\frac{p_n[W(.,t;q);q]q^{nx}}{n}\right]\\
&=(1-q)\;\exp\left[\sum\limits_{n=1}^{\infty}\frac{p_n[w(.,t;q^{\frac{1}{\sqrt{t}}});q^{\frac{1}{\sqrt{t}}}](q^{\frac{1}{\sqrt{t}}})^{n\sqrt{t}x}}{n}\right]\\
&=\frac{1-q}{1-q^{\sqrt{t}}}\left[R_{w(.,t;q^{\frac{1}{\sqrt{t}}})}(\sqrt{t}
x;q^{\frac{1}{\sqrt{t}}})\right],
\end{split}
\nonumber
\end{equation}
which is clearly equivalent to equation (\ref{R612}). The third
equation in theorem \ref{theorem541} (which is a partial
differential equation for $R_{w(.,t;q)}(x;q)$) and a change of
variables result in the partial differential equation
\begin{equation}\label{1000000}
\begin{split}
\frac{\partial}{\partial u}\left[R^2(u,t;Q)\right]-&\frac{1-Q}{\ln
Q^{-1}}u\frac{\partial R(u,t;Q)}{\partial
u}-QR(u,t;Q)-Q(1-Q)\frac{\partial R(u,t;Q)}{\partial Q}\\
&= 2t\frac{1-Q}{\ln Q^{-1}}\frac{\partial R(u,t;Q)}{\partial t},
\end{split}
\end{equation}
where $R(u,t;Q):=R_{W(.,t;Q)}(u;Q)$. Let us write the function
$R(u,t;Q)$ in the form
\begin{equation}\label{RSUM}
R(u,t;Q)=\check{R}(u;Q)+\overline{R(u,t;Q)},
\end{equation}
where $\check{R}(u;Q)$ is defined in terms of $\check{h}_n$ by
(\ref{CH}). Then
$$
\underset{t\rightarrow\infty}{\lim}\overline{R(u,t;Q)}=0,\;\;\mbox{and}\;\;
\underset{t\rightarrow\infty}{\lim}\left[t\;\frac{\partial
\overline{R(u,t;Q)}}{\partial t}\right]=0.
$$
Substituting (\ref{RSUM}) into partial differential equation
(\ref{1000000}) and taking the large $t$ limit  we obtain
$$
\frac{\partial}{\partial
u}\left[\check{R}^2(u;Q)\right]-\frac{1-Q}{\ln
Q^{-1}}u\frac{\partial\check{R}(u;Q)}{\partial
u}-Q\check{R}(u;Q)-Q(1-Q)\frac{\partial\check{R}(u;Q)}{\partial Q}=0
$$
The partial differential equation just written above is precisely
that which have appeared previously in the proof of theorem
\ref{THEOREM5222}. The proof of theorem \ref{THEOREM5222} shows that
the solution $\check{R}(u;Q)$ must satisfy the same equation as the
function $R_{\Omega(.;q)}(x;q)$ (equation (\ref{OmegaOmega})).
Therefore $\check{R}(u;Q)$ coincides with $R_{\Omega(.;q)}(x;q)$,
and the limiting moments $\check{p}_n$, $\check{h}_n$ coincide with
the corresponding moments of the diagram $\Omega(s;q)$. Thus the
normalized diagram $W(s,t;q)$ converges uniformly to $\Omega(s;q)$
as $t\rightarrow\infty$.
\end{proof}
\section{Growth of rectangular diagrams}
The aim of this section is to  relate the growth of the diagrams in
the $q$-analog of the Plancherel process, and equation (\ref{iii})
more directly.
\subsection{The definition of the growth}
Let $\left\{x_k\right\}_{k=1}^{m+1}$ and
$\left\{y_k\right\}_{k=1}^{m}$ be the points of minima and maxima of
a rectangular  diagram $w$ correspondingly, see Figure 1. Consider a
one-parameter deformation $w_t$ of $w$ by attaching a tiny square of
area $\mu_k(w;q) t$ above each minimum $x_k$. Such deformation is
referred to as \textit{the growth of a rectangular diagram.} The
interlacing sequences associated with the deformed diagram $w_t$ are
$$
x_t=\left\{x_1,y_1,\ldots,x_{m},y_{m},x_{m+1}\right\},
$$
and
$$
y_t=\left\{x_1-\sqrt{\mu_1t},x_1+\sqrt{\mu_1t},x_2-\sqrt{\mu_2t},x_2+\sqrt{\mu_2t},\ldots,
x_{m+1}-\sqrt{\mu_{m+1}t},x_{m+1}+\sqrt{\mu_{m+1}t}\right\}.
$$
Thus $x_t$ and $y_t$ defined above  are the sequences of the minima
and of the maxima of the deformed diagram $w_t$.
\subsection{The differential equation for the infinitesimal
growth}
\begin{prop}
If a rectangular diagram $w_t$ grows according to the transition
probabilities $\mu_k(w;q)$ defined by equation (\ref{WEIGHTSQQ}),
then infinitesimally (for small $t$) the function $R_{w_t}(x;q)$
evolves according to differential equation (\ref{iii}).
\end{prop}
\begin{proof}
Let $\mu_q$ be the $q$-transition measure of $w$. The
$q$-deformation of the $R$-function of the diagram $w$,
$R_{w}(x;q)$, and the $q$-deformation of the $R$-function of the
$q$-transition measure $\mu_q$ of $w$, $R_{\mu_q}(x;q)$, are given
by
$$
R_{w}(x;q)=(1-q)\;\frac{\prod_{i=1}^m\left(1-q^{x-y_i}\right)}{\prod_{i=1}^{m+1}\left(1-q^{x-x_i}\right)},
\;\;\mbox{and}\;\;
R_{\mu_q}(x;q)=(1-q)\;\sum\limits_{k=1}^{m+1}\frac{\mu_k(w;q)}{1-q^{x-x_k}}.
$$
We have
$$
R_{w}(x;q)=R_{\mu_q}(x;q).
$$
The $q$-deformation of the $R$-function of the deformed diagram
$w_t$, $R_{w_t}(x;q)$, and the $q$-deformation of the $R$-function
of the $q$-transition measure $\mu_{q,t}$ of $w_t$,
$R_{\mu_{q,t}}(x;q)$, are given by
\begin{equation}\label{43}
R_{\mu_{q;t}}(x;q)=(1-q)\sum\limits_{j=1}^{2(m+1)}\frac{\nu_j(w_t;q)}{1-q^{x-x_t^{(j)}}},
\end{equation}
\begin{equation}\label{A54}
R_{w_t}(x;q)=(1-q)\frac{\prod_{i=1}^{m+1}\left(1-q^{x-x_i}\right)
\prod_{i=1}^{m}\left(1-q^{x-y_i}\right)}{\prod_{i=1}^{m+1}\left(1-q^{x-x_i-\sqrt{\mu_it}}\right)
\left(1-q^{x-x_i+\sqrt{\mu_it}}\right)},
\end{equation}
where $x_t^{(j)}$, $j=1,\ldots, 2(m+1)$ are the elements of the set
$x_t$. The deformation preserves the equality between the
$q$-deformation of the $R$-function of the diagram, and the
$q$-deformation of the $R$-function of the corresponding
$q$-transition measure, i.e.
$$
R_{\mu_{q;t}}(x;q)=R_{w_t}(x;q).
$$
Since  $R_{\mu_{q;t}}(x;q)$ at $t=0$ must coincide with
$R_{\mu_{q}}(x;q)$ the following relations between transition
probabilities must be true
\begin{equation}\label{nununu}
\nu_1+\nu_2=\mu_1,\;\nu_3+\nu_4=\mu_2,\ldots,\nu_{2m+1}+\nu_{2m+2}=\mu_{m+1}.
\end{equation}

The functions $R_{w_t}(x;q)$ and $R_{w}(x;q)$ are related to each
other by the expression
$$
R_{w_t}(x;q)=R_{w}(x;q)
\prod\limits_{k=1}^{m+1}\left(\frac{1-q^{x-x_k-\sqrt{\mu_kt}}}{1-q^{x-x_k}}\right)^{-1}
\left(\frac{1-q^{x-x_k+\sqrt{\mu_kt}}}{1-q^{x-x_k}}\right)^{-1},
$$
which immediately follows from (\ref{43}) and (\ref{A54}). Now we
have
\begin{equation}
\begin{split}
\left(\frac{1-q^{x-x_k-\sqrt{\mu_kt}}}{1-q^{x-x_k}}\right)
\cdot\left(\frac{1-q^{x-x_k+\sqrt{\mu_kt}}}{1-q^{x-x_k}}\right)
&=\frac{1+q^{2(x-x_k)}-q^{(x-x_k)}\left(e^{\sqrt{\mu_k t}\ln
q^{-1}}+e^{-\sqrt{\mu_k t}\ln
q^{-1}}\right)}{\left(1-q^{x-x_k}\right)^2}\\
&=\frac{1+q^{2(x-x_k)}-2q^{(x-x_k)}\left(1+\frac{\mu_k t}{2}\ln^2 q^{-1}+o(t)\right)}{\left(1-q^{x-x_k}\right)^2}\\
&=1-\frac{q^{(x-x_k)}\mu_k\ln^2q^{-1}}{\left(1-q^{x-x_k}\right)^2}t+o(t).
\end{split}
\nonumber
\end{equation}
Using this we can rewrite the relation between the functions
$R_{w_t}(x;q)$ and $R_{w}(x;q)$ as follows
\begin{equation}
R_{w_t}(x;q)=R_{w}(x;q)\left[1+\sum\limits_{k=1}^{m+1}\frac{q^{x-x_k}\mu_k\ln^2q^{-1}}{\left(1-q^{x-x_k}\right)^2}
t+o(t)\right], \nonumber
\end{equation}
which clearly implies
\begin{equation}\label{A59}
\frac{\partial R_{w_t}(x;q)}{\partial
t}\biggl\vert_{t=0}=\sum\limits_{k=1}^{m+1}\frac{q^{x-x_k}\mu_k\ln^2q^{-1}}{\left(1-q^{x-x_k}\right)^2}
R_{w}(x;q).
\end{equation}
Differentiate the function  $R_{w_t}(x;q)$  with respect to $x$ and
obtain
\begin{equation}
\begin{split}
&\frac{\partial R_{w_t}(x;q) }{\partial x}=\frac{\partial
R_{\mu_{q,t}}(x;q) }{\partial
x}=(1-q)\sum\limits_{k=1}^{2(m+1)}\nu_k\frac{\partial}{\partial
x}\frac{1}{1-q^{x-x_t^{(k)}}}\\
&=(1-q)\sum\limits_{k=1}^{m+1}\nu_{2k-1}\frac{\partial}{\partial
x}\frac{1}{1-q^{x-x_k-\sqrt{\mu_kt}}}+(1-q)\sum\limits_{k=1}^{m+1}\nu_{2k}\frac{\partial}{\partial
x}\frac{1}{1-q^{x-x_k+\sqrt{\mu_kt}}}\\
&=-(1-q)\ln
q^{-1}\sum\limits_{k=1}^{m+1}\nu_{2k-1}\frac{q^{x-x_k-\sqrt{\mu_k
t}}}{\left(1-q^{x-x_k-\sqrt{\mu_kt}}\right)^2}-(1-q)\ln
q^{-1}\sum\limits_{k=1}^{m+1}\nu_{2k-1}\frac{q^{x-x_k-\sqrt{\mu_k
t}}}{\left(1-q^{x-x_k-\sqrt{\mu_kt}}\right)^2}.
\end{split}
\nonumber
\end{equation}
In particular, from the expression just written above it follows
that
\begin{equation}\label{A510}
\begin{split}
\frac{\partial R_{w_t}(x;q)}{\partial x}\biggl\vert_{t=0}&=-(1-q)\ln
q^{-1}\sum\limits_{k=1}^{m+1}\left(\nu_{2k-1}+\nu_{2k}\right)\frac{q^{x-x_k}}{(1-q^{x-x_k})^2}\\
&=-(1-q)\ln
q^{-1}\sum\limits_{k=1}^{m+1}\mu_k\frac{q^{x-x_k}}{(1-q^{x-x_k})^2},
\end{split}
\end{equation}
where in the last equation we have used relation (\ref{nununu})
between transition probabilities. We compare the righthand sides of
equations (\ref{A510}) and (\ref{A59}), and obtain differential
equation (\ref{iii}).
\end{proof}

\end{document}